\pdfoutput=1
\documentclass[a4paper,english]{jnsao}
\usepackage[utf8]{inputenc}
\usepackage{enumerate}
\usepackage[shortlabels]{enumitem}
\usepackage{comment}
\usepackage[nameinlink,capitalize]{cleveref}
\usepackage{standalone}
\usepackage{blkarray,booktabs}
\usepackage{colortbl,makecell}
\usepackage[noadjust]{cite}
\usepackage{pgfplots}
\usepackage[svgnames]{xcolor}
\usepgfplotslibrary{colorbrewer}

\pgfplotsset{
    compat=1.18,
    edge/.style = {thick, gray!50!black,forget plot},
    data/.style = {mark=*, blue, only marks},
    backdata/.style = {mark=*, blue!30!white, only marks, onlyback, forget plot},
    itergen/.style = {line width = 0.7pt},
    iter1/.style = {mark=x, Crimson, itergen},
    iter2/.style = {mark=star, BlueViolet, itergen},
    iter3/.style = {mark=asterisk, DarkOliveGreen, itergen},
    backiter1/.style = {mark=x, Crimson!70!white, onlyback, forget plot, itergen},
    backiter2/.style = {mark=star, BlueViolet!70!white, onlyback, forget plot, itergen},
    backiter3/.style = {mark=asterisk, DarkOliveGreen!70!white, onlyback, forget plot, itergen},
    origin/.style = {mark=o, green, only marks, thick},
    surfstyle/.style = {very nearly opaque, forget plot},
    legend style = {
        inner sep = 0pt,
        outer xsep = 5pt,
        outer ysep = 0pt,
        legend cell align = left,
        align = left,
        draw = none,
        fill = none,
        font = \footnotesize
    },
    illustr3d/.style = {
        axis equal,
        width=0.7\linewidth,
        height=0.7\linewidth,
        scale only axis,
        enlargelimits=false,
        colormap access=map,
        colormap/Blues,
        colorbar,
        point meta rel=axis wide,
        shader = interp,
        xlabel = {$x$},
        ylabel = {$y$},
        zlabel = {$z$},
        ticks = none,
        axis line style = {draw=none},
        legend columns = 3,
        legend style = {
            at = {(0.5, 1.1)},
            anchor = north,
            column sep = 1ex,
        },
        mark size=1.5pt,
        colorbar style = {width = 1ex, height = 0.4\linewidth},
    }
}

\newcommand{\term}{\emph}

\newcommand{\field}[1]{\mathbb{#1}}
\newcommand{\N}{\mathbb{N}}

\newcommand{\R}{\field{R}}

\newcommand{\abs}[1]{|#1|}

\newcommand{\inv}[1]{#1^{-1}}

\newcommand{\Union}\bigcup
\newcommand{\Isect}\bigcap
\newcommand{\union}\cup
\newcommand{\isect}\cap
\newcommand{\bigunion}\bigcup
\newcommand{\bigisect}\bigcap

\newcommand{\defeq}{:=}

\newcommand{\downto}{\searrow}
\newcommand{\upto}{\nearrow}

\newcommand{\overbar}[1]{\mkern 1.5mu\overline{\mkern-1.5mu#1\mkern-1.5mu}\mkern 1.5mu}

\DeclareMathOperator*{\argmin}{arg\,min}

\DeclareMathOperator{\dom}{dom}

\DeclareMathOperator{\diam}{diam}

\DeclareMathOperator{\prox}{prox}
\DeclareMathOperator{\gdesc}{g-desc}
\DeclareMathOperator{\gexp}{g-exp}
\DeclareMathOperator{\desc}{desc}

\newcommand{\iprod}[2]{\langle #1,#2\rangle}

\def \weaktostarSym{\setbox0=\hbox{$\rightharpoonup$}\rlap{\hbox 
        to\wd0{\hss\raise1ex\hbox{$\scriptscriptstyle{*\,}$}\hss}}\box0}
    
\def\this#1{#1^k}
\def\nexxt#1{#1^{k+1}}

\def\nextx{\nexxt{x}}

\def\thisx{\this{x}}

\let\phi\varphi

\let\epsilon\varepsilon

\renewrobustcmd{\downto}{{{\mathchoice%
            {\rotatebox[origin=c]{-20}{$\to$}}
            {\rotatebox[origin=c]{-20}{$\to$}}
            {\rotatebox[origin=c]{-20}{\scalebox{0.75}{$\to$}}}
            {\rotatebox[origin=c]{-20}{\scalebox{0.6}{$\to$}}}
}}}

\renewrobustcmd{\upto}{{{\mathchoice%
            {\rotatebox[origin=c]{20}{$\to$}}
            {\rotatebox[origin=c]{20}{$\to$}}
            {\rotatebox[origin=c]{20}{\scalebox{0.75}{$\to$}}}
            {\rotatebox[origin=c]{20}{\scalebox{0.6}{$\to$}}}
}}}

\usepackage[textsize=footnotesize,colorinlistoftodos]{todonotes}

\usepackage[nameinlink,capitalize]{cleveref}

\theoremstyle{definition}
\newtheorem{assumption}[theorem]{Assumption}
\crefname{assumption}{Assumption}{Assumptions}

\usepgfplotslibrary{fillbetween}
\usetikzlibrary{external}
\usepackage{tikz-3dplot}
\tikzexternaldisable
\makeatletter
\let\@the@H@page\relax
\makeatother

\title{Forward-backward splitting in bilaterally bounded Alexandrov spaces}
\shorttitle{Forward-backward splitting in Alexandrov spaces}

\date{2025-03-31 (revised 2026-04-02)}

\author{
    Heikki von Koch\thanks{%
        Department of Mathematics and Statistics, University of Helsinki, Finland,
        \email{heikki.vonkoch@helsinki.fi},
        \orcid{0009-0006-6906-9702}
    }
    \and
    Tuomo Valkonen\thanks{%
        MODEMAT Research Center in Mathematical Modeling and Optimization, Quito, Ecuador
        \emph{and}
        Department of Mathematics and Statistics, University of Helsinki, Finland,
        \emph{and}
        Escuela Politécnica Nacional, Quito, Ecuador
        \email{tuomo.valkonen@iki.fi},
        \orcid{0000-0001-6683-3572}
   }
}
\shortauthor{von Koch, Valkonen}

\acknowledgements{%
    This research has been supported by the Jane and Aatos Erkko Foundation.
}

\manuscriptcopyright{}
\manuscriptlicense{}

\manuscriptstatus{Manuscript}
\manuscripteprinttype{arxiv}
\manuscripteprint{2503.24126}

\begin{document}

\maketitle

\begin{abstract}
    With the goal of solving optimisation problems on non-Riemannian manifolds, such as geometrical surfaces with sharp edges, we develop and prove the convergence of a forward-backward method in Alexandrov spaces with curvature bounded both from above and from below.
    This bilateral boundedness is crucial for the availability of both the gradient and proximal steps, instead of just one or the other.
    We numerically demonstrate the behaviour of the proposed method on simple geometrical surfaces in $\R^3$.
\end{abstract}

\section{Introduction}

Our goal is to develop proximal gradient (i.e., forward-backward) methods for the problem
\begin{equation}
	\label{eq:intro:problem}
	\min_{x \in X}~ F(x) + G(x)
\end{equation}
in bilaterally bounded, complete, and locally compact Alexandrov spaces, in particular, on manifolds whose natural atlas may not possess a Riemannian metric, such as two-dimensional surfaces with kinks, embedded in $\R^3$. Here both $F$ and $G$ are convex with $F$ smooth, in a suitable sense that we will discuss in \cref{sec:convex}.
Alexandrov spaces, as will be introduced in \cref{sec:alexandrov}, roughly speaking, are manifold-like spaces with curvature bounded from either above or below.

\emph{No such method yet exists in the literature.} In \cite{ohta2015discretetime}, gradient methods are defined in Alexandrov spaces with lower curvature bounds, and proximal point methods are defined in spaces with upper curvature bounds.
With appropriate assumptions on the objective function and step sizes, both methods converge to a minimiser. It stands to reason that by combining the methods, we could create a proximal gradient method in \term{bilaterally bounded} Alexandrov spaces, i.e., spaces with both upper and lower curvature bounds.

The standard Hilbert space proximal gradient method was first generalised to Stiefel manifolds in \cite{stiefelproxgrad}, where the objective function was a sum of a convex function and a smooth but possibly nonconvex function. This method was then extended to smooth Riemannian manifolds \cite{riemproxgradhuang}, with both functions being retraction convex. Its linear convergence rate was further studied in \cite{riemproxgradlinear}. However, these methods are not intrinsic but are based on the total surrogate model.
On manifolds, this involves solving a subproblem in the tangent space, which can be difficult.
Such an approach could also be taken in Alexandrov spaces, see \cref{rem:total-surrogate}.
It avoids the need for an upper curvature bound, and presents a possible direction for future research.
\emph{Our proposed method, however, is fully intrinsic:} the calculation of the gradient step and proximal step are completely separate phases. \emph{This can be much easier than minimising the total surrogate model.}
An intrinsic nonconvex method using the Riemannian Kurdyka-\L{}ojasiewicz property in Hadamard manifolds was introduced in \cite{nonconvexriemproxgrad}. In the same limited Hadamard setting, a convex method similar to ours was also very recently introduced in \cite{bergmann2025intrinsicriemannianproximalgradient}. An advantage of the bilateral Alexandrov approach is that the metric induced by the natural $C^{1,1}$ distance coordinate atlas has only $C^{0,1}$ regularity, which is optimal. In contrast, standard convex optimisation arguments in smooth Riemannian geometry rely on comparison theorems that typically need the metric to be $C^2$.

Alexandrov spaces with just upper curvature bounds include Hadamard spaces, such as the manifold of symmetric positive-definite matrices. Several recent works discuss optimisation in Hadamard spaces \cite{bacak2014convex,ferreira2002proximal,bergmann2021fenchel,bento2017iteration,weinmann2014total,bergmann2016parallel,banertbackwardbackward,Jost1995nonpositivecurvature,mayergradientflowsnonpositively}.
Hadamard spaces, which have nonpositive curvature, are easy from the point of view of optimisation, as geodesics are unique and an inequality reminiscent of the Pythagoras identity in Hilbert spaces, holds in the correct direction for standard convergence proofs to go through with little modifications \cite{tuomov-firstorder}.

Informally speaking, the curvature bounds of an Alexandrov space amount to a certain degree of concavity or convexity of its distance function when compared to its counterpart in the model spaces. Geometrically this means that the triangles in a nonnegatively/nonpositively curved Alexandrov space are not thinner/thicker than their Euclidean counterparts. Even though at first glance these purely metric spaces have very little to do with the well-behaved model spaces they are compared to, it turns out that much can be said about their local and global structure. For instance, perhaps the most important shared property of all Alexandrov spaces is that of the well-defined notion of an angle. It allows us to develop differential calculus in the form of differentials and gradients which further allow us to talk about optimality conditions and gradient flow.

Earlier works on optimisation in general Alexandrov spaces \cite{ohta2015discretetime,proxsplitcatlauster}, which we review in \cref{sec:prox-and-grad}, assume the effective domain of the map $G$ to be small or, how it is written in those works, they calculate the proximal map not for just any $G$, but for $G + \delta_A$, where $A$ is assumed to have a small diameter.
This has the effect of forcing the existence of geodesics between all points of $A$, allowing for simple “convex optimisation” style proofs even in non-Hadamard spaces. We do not make that restriction.
Instead our proofs are more reminiscent of convergence proofs for nonconvex composite optimisation \cite{tuomov-nlpdhgm-redo,tuomov2024online-eit,tuomov2024tracking}, based on a priori and a posteriori estimates. The a priori estimate proves that, if our current iterate is close enough to a solution (i.e., deep enough in the interior of $A$), then the next iterate does not escape from $A$. The a posteriori estimate then proves that we do not escape at all, and even advance towards the solution.
The effect is that our results only require initialising close to solution of \eqref{eq:intro:problem}, instead of restricting the problem to a set of bounded diameter. The restriction is a common assumption in optimisation on manifolds and is necessitated by the uniqueness of geodesics between points in the set where our iterates lie; the maximum diameter of this set depends on the upper curvature bound.

The bilateral bounds of a complete locally compact Alexandrov space force the space to be a topological manifold with a continuous metric (\cref{bilat}) and complete tangent spaces. These allow us to use the exponential map instead of the rather technical gradient exponential map in the definition of the gradient descent map. This has the benefit of being easier to compute in simple examples. One could also potentially use the gradient exponential in the gradient descent map, but as the gradient exponential and the exponential maps agree on whenever the latter is defined, these result in the same map, at least, in geodesically convex sets.

After the aforementioned introductory material of \cref{sec:alexandrov,sec:convex,sec:prox-and-grad}, our main contributions start in \cref{sec:proxgrad} with a study of contractivity properties of a proximal gradient map in bilaterally bounded Alexandrov spaces. We then use this work in \cref{sec:fb} to prove the convergence of a proximal gradient method based on the iterative application of the eponymous map. We finish in \cref{sec:examples} by working out and numerically illustrating the proximal gradient map for a Lasso-type problem on the surfaces of a cube and a capped cylinder in $\R^3$.

\section{Alexandrov spaces}
\label{sec:alexandrov}

We start by introducing definitions and notation used in this paper. For a comprehensive treatment of metric geometry and basics of spaces with curvature bounded from above and below, see for example \cite{burago2001course}.
For a fundamental study of general Alexandrov spaces, see \cite{petruninfoundations}.

Let $(X,d)$ be a metric space. A continuous curve $\gamma : [0,1] \to X$ is called a unit-speed geodesic (or simply, geodesic) if for all $s, t \in [0,1]$ we have $d(\gamma(s),\gamma(t)) = |s-t|$, i.e., it is globally distance-preserving. Metric spaces $(X,d)$ that admit a unit-speed geodesic between any two points $x$ and $y$ are called geodesic, and such unit-speed geodesics from $x$ to $y$ will be denoted by $\gamma(t) = x \#_t y$ for $t \in [0,1]$ (note that the curve may not be unique). A subset $A \subset X$ of a geodesic space $X$ is called (geodesically) convex if for every two points $x,y \in A$, any unit-speed geodesic $x \#_t y$ lies in $A$.

Given $\kappa \in \R$, the model $\kappa$-plane, denoted by $\mathbb{M}^2(\kappa)$, is a complete simply-connected $2$-dimensional Riemannian manifold of constant curvature $\kappa$. This limits the metric space $(\mathbb{M}^2(\kappa),d_{\mathbb{M}^2(\kappa)})$ to be one of three possibilities (up to isometry): the hyperbolic space of constant curvature, if $\kappa < 0$; the Euclidean plane, if $\kappa = 0$; or the sphere of constant curvature, if $\kappa > 0$. Set $D_\kappa \coloneqq \diam \mathbb{M}^2(\kappa)$, where the diameter of a metric space $(X,d)$ is defined by $\diam X \coloneqq \sup_{x, y \in X} d(x,y)$. We have $D_\kappa = \infty$, if $\kappa \leq 0$ and $D_\kappa = \pi/\sqrt{\kappa}$, if $\kappa > 0$.

To define geodesic metric spaces with bounded sectional curvature, we need to compare the metric triangles defined in $X$ to those defined in a model $\kappa$-plane $\mathbb{M}^2(\kappa)$. For any three points $x,y,z \in X$ satisfying $d(x,y) + d(y,z) + d(z,x) < 2D_\kappa$, there are unique points (up to isometry, \cite[2.14 Lemma]{bridson2011metric}) $\tilde{x}, \tilde{y}, \tilde{z} \in \mathbb{M}^2(\kappa)$ such that
	\[
		d_{\mathbb{M}^2(\kappa)}(\tilde{x}, \tilde{y}) = d(x,y), \quad d_{\mathbb{M}^2(\kappa)}(\tilde{y}, \tilde{z}) = d(y,z), \quad d_{\mathbb{M}^2(\kappa)}(\tilde{z}, \tilde{x}) = d(z,x).
	\]
The triangle $\triangle \tilde{x} \tilde {y} \tilde{z}$ is called the comparison triangle of $\triangle xyz$ in $\mathbb{M}^2(\kappa)$. A unit-speed geodesic $\tilde{\gamma}: [0,1] \to \mathbb{M}^2(\kappa)$ from $\tilde{x}$ to $\tilde{y}$ will be denoted by $\tilde{x} \#_t \tilde{y}$.

\begin{definition}[Alexandrov spaces]
A geodesic metric space $(X,d)$ is called an Alexandrov space with curvature bounded from above by $\kappa \in \R$ if, for every $x,y,z \in X$ such that $d(x,y) + d(y,z) + d(z,x) < 2D_\kappa$ and every unit speed geodesic $y \#_t z$,
	\[
		d(y \#_t z, x) \leq d_{\mathbb{M}^2(\kappa)}(\tilde{y} \#_t \tilde{z}, \tilde{x});
	\]
$(X,d)$ is called an Alexandrov space with curvature bounded from below by $\kappa$ if the reverse inequality holds.
\end{definition}

We call both spaces bounded from above and spaces bounded from below by the generic name Alexandrov space. We will usually distinguish the bounds from each other by denoting the upper bound by $\overline{\kappa}$ and the lower bound by $\underline{\kappa}$. Some typical examples are given here.

\begin{example}
\begin{enumerate}[leftmargin=*, nosep]
	\item Complete Riemannian manifolds with sectional curvature at most $\overline{\kappa}$ and their Gromov-Hausdorff limits are Alexandrov spaces with curvature bounded from above by $\overline{\kappa}$. Others include locally simply connected subsets of $\R^2$, isometric gluings of spaces bounded from above along convex sets, Hadamard manifolds and Hilbert spaces (such as $\R^n$; bounded from above by $0$), and metric trees (bounded from above by any $\overline{\kappa} \in \R$).
	Away from the vertices, polyhedral surfaces in $\R^n$ are bounded from above with $\overline{\kappa}=0$.

	\item Complete Riemannian manifolds with sectional curvature at least $\underline{\kappa}$ and their Gromov-Hausdorff limits are Alexandrov spaces with curvature bounded from below by $\underline{\kappa}$. Other examples include Hilbert spaces (bounded from below by $0$), $L^2$-Wasserstein spaces over nonnegatively curved spaces, and convex hypersurfaces of Riemannian manifolds (such as graphs of convex functions of the form $F: \R^n \to \R$ and boundaries of closed convex bodies in $\R^n$). Importantly, convex polyhedral surfaces in $\R^n$ are bounded from below by $0$ and the tangent spaces of non-vertex points are isometric to $\R^n$.
    Equipped with an appropriate distortion distance, the space of metric measure spaces is also an Alexandrov space with curvature bounded from below \cite{sturm2023space}. Such spaces have relevance to, e.g., geometric shape interpolation \cite{beier2024tangentialfixpointiterationsgromovwasserstein}.

	\item Spaces with bilaterally bounded sectional curvature in the sense of Alexandrov include Hilbert spaces (bounded from above and below by $0$) and, for example, a finitely long flat cylinder with hemispheres isometrically glued to its ends (bounded from below by $0$ and above by the (constant) curvature of the sphere). In fact, spaces with bilateral curvature bounds are Gromov-Hausdorff limit spaces of smooth Riemannian manifolds with uniformly bounded bilateral sectional curvature.
\end{enumerate}
\end{example}

To study the local geometry of Alexandrov spaces, we need to replace the concept of tangent spaces from the theory of smooth manifolds with the notion of a space of directions. Fix a point $x \in X$ and two unit-speed geodesics starting from it: $\gamma(0) = x = \eta(0)$. The angle between $\gamma$ and $\eta$ is defined as
	\[
		\angle_x(\gamma, \eta) \coloneqq \lim_{t,s \downto 0} \angle \tilde{\gamma}(t) \tilde{x} \tilde{\eta}(s),
	\]
where $\angle \tilde{\gamma}(t) \tilde{x} \tilde{\eta}(s) \in$ $[0, \pi]$ is the well-defined angle at $\tilde{x}$ of a comparison triangle $\triangle \tilde{\gamma}(t) \tilde{x} \tilde{\eta}(s)$ in $\mathbb{M}^2(\kappa)$. The angle always exists and is independent of the choice of $\kappa \in \R$ (this is due to the angles being proportional to the area of the triangle, which in turn can be bounded using its perimeter; see \cite[Lemma 6.4]{petruninfoundations}). Setting the geodesics $\gamma$ and $\eta$ to be equivalent if $\angle_x(\gamma, \eta) = 0$ defines an equivalence class and a pseudometric space $(\Sigma^{\prime}_{x}X,\angle_x)$, where $\Sigma^{\prime}_{x}X$ is the set of all nontrivial unit-speed geodesics (strictly speaking, equivalence classes of unit-speed geodesics) starting from $x$. The completion of $\Sigma^{\prime}_{x}X$ is called the space of directions at $x$ and is denoted by $\Sigma_{x}X$.

\begin{definition}[The tangent cone]
The tangent cone $(T_xX, d_{\angle_x})$ at $x \in X$ is defined as the Euclidean cone over $(\Sigma_{x}X, \angle_x)$, i.e., $T_xX \coloneqq (\Sigma_{x}X \times [0,\infty))/\sim$, where $(\gamma,t) \sim (\eta,s)$ if and only if $t=s=0$. The metric on $T_xX$ is defined as
	\[
		d_{\angle_x}((\gamma,t), (\eta,s)) \coloneqq \sqrt{t^2 + s^2 - 2ts\cos \angle_x(\gamma, \eta)}
	\]
for $(\gamma,t), (\eta,s) \in T_xX$.
\end{definition}

Similarly, the Euclidean cone over the pseudometric space $(\Sigma^{\prime}_{x}X,\angle_x)$ is denoted by $(T^{\prime}_xX, d_{\angle_x})$. The definition of the tangent cone metric allows us to define an “inner product” of two elements $(\gamma,t), (\eta,s) \in T_xX$ as
\begin{equation}
	\label{innprod}
	\langle (\gamma,t),(\eta,s) \rangle \coloneqq ts \cos \angle_x (\gamma, \eta) = \frac{t^2 + s^2 - d^2_{\angle_x}((\gamma,t), (\eta,s))}{2},
\end{equation}
where $\langle (\gamma,t),(\eta,s) \rangle \coloneqq 0$, if $t=0$ or $s=0$.

The elements of the tangent cone $T_xX$ will be called tangent vectors (despite the fact that it is a cone and not necessarily a vector space). A vector $(\gamma, t) \in T^{\prime}_xX$ with $t=d(x,y)$, corresponding to a geodesic $\gamma(t) = x\#_ty$, is called logarithm (of the geodesic $\gamma$) and is denoted by $\log_x y$ (not necessarily uniquely defined). In general, we have the following definition for velocity of curves.

\begin{definition}[{\cite[Definition 6.9]{petruninfoundations}}]
	Let $X$ be a metric space, $a>0$, and $\alpha: [0,a) \to X$ be a function, not necessarily continuous, such that $\alpha(0)=x$. We say that $v \in T_xX$ is the right derivative of $\alpha$ at $0$, briefly $\alpha^+(0)=v$, if for some (and therefore any) sequence of vectors $v_n \in T^{\prime}_xX$ such that $v_n \to v$ as $n \to \infty$, and corresponding geodesics $\gamma_n$, we have
		\[
			\limsup_{t \downto 0} \frac{d(\alpha(t), \gamma_n(t))}{t} \to 0 \quad \text{as} \quad n \to \infty.
		\]
\end{definition}

Especially, if we choose $\alpha$ to be a geodesic $\gamma(t) = x\#_t y$ and $v_n=v$ for all $n$, then $\gamma^+(0) = v = \log_x y \in T^{\prime}_xX$.

\begin{remark}
	For a Riemannian manifold $(X,d)$, the space of directions $(\Sigma_{x}X, \angle_x)$ and the tangent cone $(T_xX, d_{\angle_x})$ coincide with the unit tangent sphere and the tangent space at $x$, respectively \cite[Section 1.4]{yuburagobelow}.
\end{remark}

The definition of the angle allows for the following fundamental rule for differentiating the length of a variable curve; see \cite[8.42, 9.36]{petruninfoundations}.

\begin{theorem}[First variation formula]
	\label{firstvar}
	Let $(X,d)$ be an Alexandrov space with curvature bounded from above or below by $\kappa$, and assume that $X$ is locally compact when the curvature is bounded from below. For any point $z \in X$ and geodesic $\gamma(t) = x \#_t y$ with $x, y \in X$, and $0 < d(x,z) < \pi / \sqrt{\kappa}$ if $\kappa>0$, we have
	\[
		\lim_{t \downto 0} \frac{d_z(\gamma(t)) - d_z(x)}{t} = - \cos \theta_{\min},
	\]
	where $d_z: X \to \R : a \to d(z,a)$ and $\theta_{\min}$ is the minimum of angles between $\gamma$ and all geodesics from $x$ to $z$.
\end{theorem}

\begin{remark}
	The first variation inequality “$\leq$” is very general as it holds for metric spaces with defined angles \cite[6.7]{petruninfoundations} (no need for curvature bounds or compactness). This turns out to be sufficient for the trigonometric bound for the squared distance function (\cref{triglower}), whereas the other trigonometric bound (\cref{trigupper}) requires the full equality.
\end{remark}

It turns out that bilaterally bounded Alexandrov spaces have some differentiable structure of low smoothness. The dimension mentioned is the topological dimension as is implied by the use of coordinate maps.

\begin{theorem}[{\cite[Theorem 5]{berestovskii}}]
	\label{bilat}
	Let $(X,d)$ be a complete locally compact Alexandrov space with curvature bounded above by $\overline{\kappa}$ and below by $\underline{\kappa}$ ($\underline{\kappa} \leq \overline{\kappa}$). Then $X$ is finite-dimensional and admits a natural atlas of distance coordinates of class $C^1$. Furthermore, the distance on $X$ is defined by a metric of class $C^0$ (continuous).
\end{theorem}

In fact, from the construction of the distance coordinates, a little more is true; the atlas is $C^{1,1}$-smooth making the components of the metric tensor $C^{0,1}$-smooth. No more regularity can be expected, however, as can be seen from the classic example of gluing a hemisphere to the top of a flat cylinder: the curvature is not continuous and hence the atlas cannot be $C^2$-smooth (such a surface is called a capsule when both ends of the cylinder are glued shut with hemispheres).

As an aside, it is important to remember that even though the distance coordinate atlas can be refined to be $C^r$ for any $1< r \leq \infty$ (due to Whitney), doing so will also alter the metric and, consequently, the behaviour of the whole algorithm altogether. It is therefore imperative to invoke the Alexandrov framework to study algorithms based on natural atlases and their metrics.

\begin{remark}
	The notion of a boundary of a finite-dimensional Alexandrov space (denoted by $\partial X$) with curvature bounded from below is required for the results regarding the gradient descent map. The boundary is defined inductively with respect to dimension and space of directions, and can be found in, for example, \cite[7.19]{yuburagobelow}. However, since we will be assuming bilateral bounds for our results, the space will have the structure of a differentiable manifold (of low regularity) and as such the boundary is defined in the standard manner with upper half-spaces.
\end{remark}

We also need the following concept of sequences satisfying a monotonicity property with an error term to analyse the convergence of optimisation algorithms in generic metric spaces. We denote $\N \coloneqq \{ 0,1,\dots \}$.

\begin{definition}
	A sequence $\{x^k\}_{k \in \N}$ in a complete metric space $(X,d)$ is called quasi-Fej{\'e}r with respect to a set $U \subset X$ if, for every $\bar{x} \in U$, there exists a sequence $\{\epsilon_k\}_{k \in \N} \subset \R$ such that $\epsilon_k \geq 0$, $\sum_{k=0}^\infty \epsilon_k < \infty$, and $d^2(x^{k+1}, \bar{x}) \leq d^2(x^k, \bar{x}) + \epsilon_k$, for all $k=0,1,\dots$.
\end{definition}

If, for every point $\bar{x} \in U$, we can choose the zero sequence $\epsilon_k = 0$ for all $k = 0,1,\dots$, then the sequence $\{x^k\}_{k \in \N}$ is simply called Fej{\'e}r (with respect to $U$). The fundamental property of quasi-Fej{\'e}r sequences is the following extension of Opial's lemma to metric spaces.

\begin{theorem}
	\label{fejer}
	Let $\{x^k\}_{k \in \N}$ be a sequence in a complete locally compact metric space $(X,d)$. If $\{x^k\}_{k \in \N}$ is quasi-Fej{\'e}r with respect to a nonempty set $U \subset X$, then $\{x^k\}_{k \in \N}$ is bounded. If furthermore, an accumulation point $x^*$ of $\{x^k\}_{k \in \N}$ belongs to $U$, then $\lim_{k \to \infty} x^k = x^*$.
\end{theorem}

\begin{proof}
Let $\{x^k\}_{k \in \N}$ be a quasi-Fej{\'e}r sequence with respect to a nonempty set $U \subset X$ and $\bar{x} \in U$. By definition, we have $d^2(x^1,\bar{x}) \leq d^2(x^0,\bar{x}) + \epsilon_0$ so that $d^2(x^2,\bar{x}) \leq d^2(x^1,\bar{x}) + \epsilon_1 \leq d^2(x^0,\bar{x}) + \epsilon_0 + \epsilon_1$. Iterating this further we have that $d^2(x^k,\bar{x}) \leq d^2(x^0,\bar{x}) + \sum_{i=0}^{k-1}\epsilon_i \leq d^2(x^0,\bar{x}) + \sum_{i=0}^{\infty}\epsilon_i < \infty$ for all $k=0,1,\dots$ by the definition of the sequence $\{\epsilon_k\}_{k \in \N}$; that is, $\{x^k\}_{k \in \N}$ is bounded.

Now let $x^* \in U$ be an accumulation point of $\{x^k\}_{k \in \N}$. Since the sequence is bounded, by local compactness, it has a subsequence $\{ x^{k_i} \}_{i \in \N}$ such that $\lim_{i \to \infty} x_{k_i}= x^*$. Let $\delta > 0$. By definition, there exists an $N_1$ such that $\sum_{i = N_1}^{\infty}\epsilon_i < \delta/2$ and an $N_2 \geq N_1$ such that $d^2(x^{k_j},x^*) < \delta/2$ for any $j \geq N_2$. From these we get that
	\[
		d^2(x^k, x^*) \leq d^2(x^{k_j}, x^*) + \sum_{i = j}^{k-1}\epsilon_i \leq d^2(x^{k_j}, x^*) + \sum_{i = j}^{\infty}\epsilon_i < \frac{\delta}{2} + \frac{\delta}{2} = \delta,
	\]
for any $k \geq j$. Since $\delta > 0$ was arbitrary, $\lim_{k \to \infty} x^k = x^*$.
\end{proof}

\section{Convex analysis in Alexandrov spaces}
\label{sec:convex}

Let $(X,d)$ be an Alexandrov space with curvature bounded from above or below by $\kappa \in \R$. We define extended real-valued functions as $F: X \to \overbar{\R} \coloneqq \R \cup \{\infty\}$ (the codomain excludes the value $-\infty$). As it is natural to consider extended real-valued functions in convex analysis, it is also useful to specify the set in which the function attains finite values. These points constitute the effective domain of a function $F$, denoted by $\dom F$, i.e.,
\[
    \dom F \coloneqq \{ x \in X \, | \, F(x) < \infty \}.
\]
If $\dom F \neq \emptyset$, the function $F$ is called proper.

Extended real-valued $\lambda$-convex functions are defined in the usual manner.

\begin{definition}[$\lambda$-convexity]
	A function $F: X \to \overbar{\R}$ is called $\lambda$-convex for $\lambda \in \R$ if
	\[
		F(\gamma(t)) \leq (1-t)F(x) + tF(y) - \frac{\lambda}{2}t(1-t)d^2(x,y)
	\]
	holds for any $x, y \in X$, $t \in [0,1]$ and any geodesic $\gamma : [0,1] \to X$ with $\gamma(t) = x \#_t y$.
\end{definition}

Note that geodesics in Alexandrov spaces are always assumed to be minimal (distance minimising) and unit-speed (parametrised by arc length). Functions are called convex and strongly convex when $\lambda \geq 0$ and $\lambda>0$, respectively. It should also be noted that functions that are strongly convex are also strictly convex (the converse is not true however). Perhaps the most important $\lambda$-convex functions are the squared distance functions. The following proposition is an explicit formulation of work known in \cite[Chapter 2, Section 5]{Berestovskij1993multidimensional} and references therein. We define the closed metric ball of radius $r>0$ around a point $x \in X$ as: $\widebar{B}(x,r) \coloneqq \{ y \in X \, | \, d(x,y) \leq r \}$.

\begin{proposition}[{\cite[Proposition 3.1]{ohtaconvexities}}]
	\label{dsquared}
	Let $(X,d)$ be an Alexandrov space with curvature bounded from above by $\overline{\kappa} > 0$. Then, for any $a \in X$, the function $d^2_a$  is $\lambda$-convex on the geodesically convex metric ball $\widebar{B}(a,r)$ with $2r = (\pi/2 - \epsilon)/\sqrt{\overline{\kappa}}$ and $\lambda = (\pi - 2\epsilon) \tan \epsilon$ for arbitrary $\epsilon \in (0,\pi/2)$.
\end{proposition}

When the space is an Alexandrov space with curvature bounded from above by $0$, the squared distance function is $2$-convex \cite[Corollary 9.26]{petruninfoundations} (in fact, these are equivalent). Moreover, since a complete Alexandrov space bounded from above by $\overline{\kappa}$, is an Alexandrov space bounded from above by $\overline{K}$, for $\overline{\kappa} < \overline{K}$ \cite[Corollary 9.18]{petruninfoundations}, we have a strong convexity factor for the squared distance function for all $\overline{\kappa} \in \R$ ($\lambda$ is given by \cref{dsquared} when $\overline{\kappa} > 0$, and $\lambda = 2$, when $\overline{\kappa} \leq 0$).

\begin{proposition}[{\cite[Lemma 3.3]{ohtawasserstein}}]
	\label{minusdsquared}
	Let $(X,d)$ be an Alexandrov space with curvature bounded from below by $\underline{\kappa} < 0$. Then, for any $a \in X$, the function $-d^2_a$ is $\lambda$-convex on the metric ball $\widebar{B}(a,r)$ with $\lambda = -2(1 - \underline{\kappa}(2r)^2)$ for all $r>0$.
\end{proposition}

When the space is a complete Alexandrov space with curvature bounded from below by $0$, the negative squared distance function is $-2$-convex \cite[Corollary 8.24]{petruninfoundations} (in fact, these are equivalent). Moreover, since a complete Alexandrov space bounded from below by $\underline{\kappa}$, is an Alexandrov space bounded from below by $\underline{K}$, for $\underline{K} < \underline{\kappa}$ \cite[Corollary 8.33]{petruninfoundations}, we have a strong concavity factor for the squared distance function for all $\underline{\kappa} \in \R$ ($\lambda$ is given by \cref{minusdsquared} when $\underline{\kappa} < 0$, and $\lambda = -2$, when $\underline{\kappa} \geq 0$).

Due to the existence of angles in Alexandrov spaces, functions have well-defined directional derivatives.

\begin{definition}[Directional derivative and differential]\cite[Definition 3.2]{ohta2015discretetime}
	Let $X$ be an Alexandrov space and $F: X \to \overbar{\R}$. The directional derivative of $F$ at $x \in$ dom $F$ in the direction $h \in T_xX$ is defined as
	\[
		D_xF(h) \coloneqq \liminf_{(\gamma, s) \to h} \bigg\{\lim_{t \downto 0} \frac{F(\gamma(st)) - F(x)}{t}\bigg\},
	\]
where $(\gamma, s) \in \Sigma^{\prime}_{x}X \times [0,\infty) \subset T_xX$. The function $D_xF : T_xX \to \R$ is called the differential of $F$ at $x$.
\end{definition}

This limit can be shown to exist for proper, $\lambda$-convex, and lower semicontinuous functions \cite[4.1]{ohtawasserstein}, while its uniqueness is established in \cite[Proposition 6.16]{petruninfoundations} for functions with additional Lipschitz continuity. For the former functions, it immediately follows from the definition of $\lambda$-convexity that
\begin{equation}
	\label{subdiff}
	D_xF(\gamma^+(0)) \leq F(y) - F(x) - \frac{\lambda}{2}d^2(x,y),
\end{equation}
for a geodesic $\gamma(t) = x \#_t y$ and any $y \in X$. This immediately yields a Fermat principle for convex functions.

\begin{corollary}[Fermat principle]
	\label{fermat}
	Let $(X,d)$ be an Alexandrov space and $F : X \to \overbar{\R}$ proper, lower semicontinuous, and $\lambda$-convex with $\lambda \geq 0$.
    Furthermore, let  $x^* \in$ \emph{dom} $F$ and $\gamma$ be any geodesic with $\gamma(0) = x^*$. The following are equivalent:
\begin{enumerate}
	\item $D_{x^*}F(\gamma^+(0)) \geq 0$, for all $\gamma^+(0) \in T_{x^*}X$;
	\item $F(x^*) = \min_{x \in X} F(x)$.
\end{enumerate}
\end{corollary}

\begin{proof}
Let $D_{x^*}F(\gamma^+(0)) \geq 0$, for any geodesic $\gamma(t) = x^* \#_t y$. Directly from \cref{subdiff} and $\lambda \geq 0$ we get
	\[
		0 \leq D_{x^*}F(\gamma^+(0)) \leq F(y) - F(x^*) - \frac{\lambda}{2}d^2(x^*,y) \leq F(y) - F(x^*) \quad \text{for all} \; y \in X,
	\]
i.e., $F(x^*) \leq F(y)$ for all $y \in X$.

Now let $F(x^*) = \min_{x \in X} F(x)$. This implies $F(\gamma(t)) \geq F(x^*)$, for any $y \in X$ with $\gamma(t) = x^* \#_t y$ and $t \in [0,1]$; that is
	\[
		\frac{F(\gamma(t)) - F(x^*)}{t} \geq 0,
	\]
for small $t>0$. Taking $t \downto 0$ yields the claim.
\end{proof}

We also need the following definition of the descending absolute gradient of $F$.

\begin{definition}[Absolute gradients]\cite[Definition 3.1]{ohta2015discretetime}
	The descending absolute gradient of $F: X \to \overbar{\R}$ at $x \in$ dom $F$ is defined by
	\[
		|\nabla_-F|(x) \coloneqq \max \bigg\{ 0, \limsup_{y \to x} \frac{F(x)-F(y)}{d(x,y)} \bigg\}.
	\]
\end{definition}

It holds that $|\nabla_-F|(x) \in [0,\infty]$ and $|\nabla_-F|(x) \leq L$ if $F$ is $L$-Lipschitz. A point $x$ is called critical if $|\nabla_-F|(x) = 0$, which holds if $F(x) = \inf_{y \in X} F(y)$.

It can be shown in Alexandrov spaces bounded from below, that for a lower semicontinuous, $\lambda$-convex function $F$ with $0 < |\nabla_-F|(x) < \infty$, there exists a unique direction $\alpha \in \Sigma_{x}X$ such that $D_xF((\alpha,1)) = -|\nabla_-F|(x)$ and
\begin{equation}
	\label{optimalitygrad}
	D_xF((\beta,1)) \geq -|\nabla_-F|(x) \langle (\alpha,1), (\beta,1) \rangle
\end{equation}
for all $\beta \in \Sigma_{x}X$ \cite[Lemma 4.3]{ohtawasserstein}. As such, the tangent vector
\begin{equation}
	\label{nabla}
	\nabla (-F)(x) \coloneqq (\alpha, |\nabla_-F|(x)) \in T_xX
\end{equation}
can be thought as a subgradient vector of $-F$ at $x$. This will be used in the definition of the gradient descent map in the lower curvature bound case. Normally, we would use the descent directions of $F$ but since the tangent cone $T_xX$ is not (in general) a vector space, the negative of the gradient may not exist. Instead, we will use a subgradient of negative $F$ .

\subsection{Trigonometric bounds}
\label{subsec:trig}

Using \cref{minusdsquared,dsquared}, we can now derive some very useful trigonometric bounds related to the squared distance functions; effectively generalisations of the law of cosine. By \cref{minusdsquared}, in a complete Alexandrov space with curvature bounded from below by $\underline{\kappa} < 0$ (note that for $\underline{\kappa} \geq 0$, we can choose $\lambda = -2$ as is mentioned after \cref{minusdsquared}), the negative distance squared function is $\lambda$-convex with $\lambda < 0$:
	\[
		-d^2(z,\gamma(t)) \leq -(1-t)d^2(z,x) - td^2(z,y) - \frac{\lambda}{2}t(1-t)d^2(x,y),
	\]
where $\gamma(t) = x \#_t y$ is a geodesic and $z \in X$ with $d(z,x)>0$. This gives us
	\[
		-\frac{d^2(z,\gamma(t)) - d^2(z,x)}{t} \leq d^2(z,x) - d^2(z,y) - \frac{\lambda}{2}(1-t)d^2(x,y),
	\]
which by taking $t \downto 0$ yields
	\[
		-D_x(d^2(z,x))(\gamma^+(0)) \leq d^2(z,x) - d^2(z,y) - \frac{\lambda}{2}d^2(x,y).
	\]
The differential exists (as $d^2_z$ is also proper and continuous) and the first variation formula \cref{firstvar} gives $D_x(d^2(z,x))(\gamma^+(0)) \leq -2 \langle \gamma^+(0), \eta^+(0) \rangle$, where $\eta : [0,1] \to X$ is the geodesic from $x$ to $z$ that realises the minimum angle between $\gamma$ and all geodesics from $x$ to $z$. The definition of the inner product \cref{innprod} further gives $-2 \langle \gamma^+(0), \eta^+(0) \rangle = -2d(x,y)d(x,z) \cos \angle_x (\gamma,\eta)$ so by rearranging we finally get
	\[
		d^2(z,y) \leq d^2(z,x) - \frac{\lambda}{2}d^2(x,y) - 2d(x,y)d(x,z) \cos \angle_x (\gamma,\eta),
	\]
which is, in fact, nothing but a modified cosine rule in disguise (since $\lambda <0$). The factor $\lambda/2$ can be given a more explicit expression as has been done in \cite{zhangfirstorder}.

\begin{lemma}[{\cite[Lemma 6]{zhangfirstorder}}]
	\label{triglower}
	If $a, b, c$ are the sides (i.e., side lengths) of a geodesic triangle in an Alexandrov space with curvature bounded below by $\underline{\kappa} \in \R$, and $\theta$ is the angle between sides $b$ and $c$, then
	\[
		a^2 \leq \frac{\sqrt{|\underline{\kappa}|}c}{\tanh(\sqrt{|\underline{\kappa}|}c)} b^2 + c^2 - 2bc \cos(\theta).
	\]
\end{lemma}

That is, the factor $-\lambda/2$ can be chosen to be $\sqrt{|\underline{\kappa}|}d(x,z)/\tanh(\sqrt{|\underline{\kappa}|}d(x,z))$. Keeping the curvature $\underline{\kappa} < 0$ fixed, we see that $\sqrt{-\underline{\kappa}}d(x,z)/\tanh(\sqrt{-\underline{\kappa}}d(x,z))$ is an increasing function of the side length $c = d(x,z)$. In particular, if $x,z \in A \subset X$ with $\diam A < \infty$, we have that $c \leq \diam A$ and $\sqrt{-\underline{\kappa}}d(x,z)/\tanh(\sqrt{-\underline{\kappa}}d(x,z)) \leq \lambda_l/2$ for
\begin{equation}
	\label{triglowlambda}
	\frac{\lambda_l}{2} \coloneqq \frac{\sqrt{-\underline{\kappa}}\diam A}{\tanh(\sqrt{-\underline{\kappa}}\diam A)}.
\end{equation}

\begin{remark}
	Although the lemma is stated for all $\underline{\kappa} \in \R$, we will only use it for $\underline{\kappa} < 0$ and choose $\lambda_l = 2$, when $\underline{\kappa} \geq 0$.
\end{remark}

Repeating the same arguments with \cref{dsquared}, in the case of a complete Alexandrov spaces with curvature bounded from above by $\overline{\kappa} > 0$, we get
\begin{equation}
    \label{eq:pythagoras-replacement}
    d^2(z,y) \geq d^2(z,x) + \frac{\lambda}{2}d^2(x,y) - 2d(x,y)d(x,z) \cos \angle_x (\gamma,\eta),
\end{equation}
where $\lambda > 0$. The factor $\lambda/2$ can be chosen to be $\sqrt{\overline{\kappa}}d(x,z)/\tan(\sqrt{\overline{\kappa}}d(x,z))$, for $\overline{\kappa} > 0$. (As mentioned after \cref{dsquared}, for $\overline{\kappa} \leq 0$, we can choose $\lambda = 2$.)

\begin{lemma}
	\label{trigupper}
	Let $(X,d)$ be an Alexandrov space with curvature bounded from above by $\overline{\kappa} > 0$ and $a, b, c$ the sides (i.e., side lengths) of a geodesic triangle with vertices and sides in a geodesically convex set $A \subset X$ with $\diam A < \pi/(2\sqrt{\overline{\kappa}})$. For the angle $\theta$ between sides $b$ and $c$ we have
	\[
		a^2 \geq \frac{\sqrt{\overline{\kappa}}c}{\tan(\sqrt{\overline{\kappa}}c)} b^2 + c^2 - 2bc \cos(\theta).
	\]
	For curvature bounded above by $\overline{\kappa} \leq 0$, we can choose $(\sqrt{\overline{\kappa}}c)/\tan(\sqrt{\overline{\kappa}}c) = 1$.
\end{lemma}

\begin{proof}
	From the previous discussion, we only need to show that the factor $\lambda/2$ can be chosen to be $\sqrt{\overline{\kappa}}c/\tan(\sqrt{\overline{\kappa}}c)$, for $\overline{\kappa} > 0$ and some side length $c$, with $c < \pi/(2\sqrt{\overline{\kappa}})$. By \cref{dsquared}, we have $\epsilon = \pi/2 - 2r\sqrt{\overline{\kappa}}$ and
		\[
			\frac{\lambda}{2} = \frac{(\pi - 2\epsilon) \tan \epsilon}{2} = 2r\sqrt{\overline{\kappa}} \tan (\frac{\pi}{2} - 2r\sqrt{\overline{\kappa}}) = \frac{2r\sqrt{\overline{\kappa}}}{\tan (2r\sqrt{\overline{\kappa}})},
		\]
	for ${\overline{\kappa}}>0$ and an arbitrary $\epsilon \in (0,\pi/2)$. Since for any $\epsilon \in (0,\pi/2)$ we have $2r < \pi/(2\sqrt{\overline{\kappa}})$, the claim follows.
\end{proof}

Again, keeping the curvature $\overline{\kappa}$ fixed, we see that $\sqrt{\overline{\kappa}}c/\tan(\sqrt{\overline{\kappa}}c)$ is a decreasing function of the side length $c = d(x,z)$. That is, if $x,z \in A \subset X$ with $\diam A < \pi/(2\sqrt{\overline{\kappa}})$, we have that $c \leq$ $\diam A$ and $\sqrt{\overline{\kappa}}c/\tan(\sqrt{\overline{\kappa}}c) \geq \lambda_u/2$ for
\begin{equation}
	\label{triguplambda}
	\frac{\lambda_u}{2} \coloneqq \frac{\sqrt{\overline{\kappa}}\diam A}{\tan(\sqrt{\overline{\kappa}}\diam A)}.
\end{equation}

These bounds allow us to write the key lemmas related to the proximal, gradient, and proximal gradient methods in an explicit manner.

\begin{remark}
	The assumption on the diameter of the subset is an important standard assumption in convex analysis in positively curved spaces as it forces the set to be convex \cite[Corollary 9.27]{petruninfoundations}. In the case of Riemannian manifolds (of sufficient smoothness), this also means that the exponential map as well as its inverse is uniquely defined in the set.
    Practically, convergence may be observed in a larger set.
\end{remark}

\section{Proximal and gradient descent maps}
\label{sec:prox-and-grad}

To analyse the proximal gradient map, we need to analyse the proximal and gradient descent maps seperately first. It turns out that the curvature bound determines which method is more appropriate for the analysis of the gradient flow; the proximal map behaves reasonabley in spaces with upper curvature bounds, while the gradient descent map is suitable in the lower bounded case. This would also imply that a bilateral bounded Alexandrov space is appropriate for the analysis of the proximal gradient map.

\subsection{The proximal map}

Let $(X,d)$ be a complete Alexandrov space with curvature bounded from above by $\overline{\kappa} \in \R$ and $G: X \to \overbar{\R}$ a proper, convex, and lower semicontinuous function. Fix the step size $\tau > 0$ and a closed, (geodesically) convex set $A \subset X$ ($\diam A < \pi/(2\sqrt{\overline{\kappa}})$, if $\overline{\kappa} > 0$) containing a nonempty sublevel set of $G$.

\begin{definition}[Constrained proximal map {\cite[Definition 4.1]{ohta2015discretetime}}]
	\label{def:proxconstrained}
	For each $x \in X$, we define the constrained proximal map
	\[
		\prox_{\tau G}^A : X \to X, \qquad \prox_{\tau G}^A(x) \coloneqq \argmin_{y \in A} \bigg\{G(y) + \frac{1}{2\tau}d^2(x,y) \bigg\}.
	\]
\end{definition}

\begin{remark}
	It is important to note the nonstandard restriction to the set $A$ in the definition of the proximal map. We will in \cref{sec:proxgrad} instead work with standard unrestricted proximal maps.
    To remedy the lack of this restriction in our convergence proofs, we will instead perform locality analysis and restrict the step sizes to not stray too far from the reference points. We will do this for the forward-backward method but one could just as well do it for the standard proximal map (with just an upper curvature bound and without the gradient step) and get similar results.
\end{remark}

\begin{remark}
	In nonpositively curved metric spaces, the proximal map and its corresponding evaluation map were first studied in \cite{Jost1995nonpositivecurvature} (here called Moreau-Yosida approximation) and independently in \cite{mayergradientflowsnonpositively}. In particular, they show that the method converges to a minimiser of $G$. The theory was further developed in \cite{jostnonlineardirichlet}. For consistency, we use the formulation used in \cite{ohta2015discretetime}.
\end{remark}

It can be shown that the proximal map is uniquely determined if $G$ is $\lambda$-convex for $\lambda \geq 0$. This follows from the strong convexity of the squared distance function ($\lambda$-convex for $\lambda > 0$) and hence from the strong convexity of the whole expression. The following fundamental lemma is key in the analysis of the proximal point method.

\begin{lemma}[{\cite[Lemma 4.6(I)]{ohta2015discretetime}}]
    \label{key1}
    Let $(X,d)$ be a complete Alexandrov space with curvature bounded from above by $\overline{\kappa} > 0$ and $\diam A < \pi/(2\sqrt{\overline{\kappa}})$. Then we have
	\[
		d^2(y, \prox_{\tau G}^A(x)) \leq d^2(y,x) + 2\tau (G(y) - G(\prox_{\tau G}^A(x)))
        \quad\text{for all}\quad
        x, y \in A.
    \]
\end{lemma}

We can consider the algorithm generated by the proximal map by taking an arbitrary starting point $x^0 \in A$ and iterating
\begin{equation}
	\label{proxiter}
	x^{k+1} = \prox_{\tau_k G}^A(x^k), \quad k \geq 0,
\end{equation}
for an appropriate positive sequence $\{ \tau_k \}_{k \in \N}$ of step sizes.
Using \cref{key1}, the convergence of the function values $G(x^k)$ to the infimum is shown in \cite[Theorem 5.1]{ohta2015discretetime} for a positive step size sequence $\{ \tau_k \}_{k \in \N}$ with $\sum_{k=0}^\infty \tau_k = \infty$. Let us also show that the sequence of points with the same step sizes converges to a minimiser of $G$ when $X$ is assumed locally compact.

\begin{theorem}
    Let $(X,d)$ be a complete locally compact Alexandrov space with curvature bounded from above by $\overline{\kappa} \in \R$. Let $G$ be lower semicontinuous and $\{ x^k \}_{k \in \N}$ be the sequence generated by \cref{proxiter} with $\sum_{k=0}^\infty \tau_k = \infty$. If the minimiser set $U^*$ of $G$ is nonempty, then $\lim_{k \to \infty} x^k = x^*$ and $x^* \in U^*$.
\end{theorem}

\begin{proof}
    Since $U^*$ is nonempty, we have $G(y) \leq G(x^k)$ for any $y \in U^*$ and $k \geq 0$. By \cref{key1} we have that $d^2(y,x^{k+1}) \leq d^2(y,x^{k})$ and hence $d(y,x^{k+1}) \leq d(y,x^{k})$ for all $k=0,1,\ldots$, i.e., $\{ x^k \}_{k \in \N}$ is a Fej{\'e}r sequence with respect to the nonempty set $U^* \subset X$. By \cref{fejer}, Fej{\'e}r sequences are bounded and since $X$ is locally compact and complete, the sequence $\{ x^k \}_{k \in \N}$ contains a convergent subsequence $\{ x^{k_i} \}_{i \in \N}$ such that $ \lim_{i \to \infty} x^{k_i} = x^* $. By lower semicontinuity and \cite[Theorem 5.1]{ohta2015discretetime} we have $G(x^*) \leq \liminf_{i \to \infty} G(x^{k_i}) = \inf_{x \in X} G(x)$, which implies that $x^* \in U$. \cref{fejer} now gives $\lim_{k \to \infty} x^k = x^*$.
\end{proof}

Without local compactness but assuming strong convexity from $G$, the algorithm can also be shown to converge to the unique minimiser of $G$ \cite[Theorem 5.7]{ohta2015discretetime}. For completeness, let us also state an iteration-complexity bound for the method.

\begin{proposition}
    Let $\{ x^k \}_{k \in \N}$ be the sequence generated by \cref{proxiter} with $\tau_k \geq \tau > 0$, for $k=0,1,\ldots$. Then for every $N \in \N$, and $x^*$ a point where $G$ achieves a minimum, we have
    \[
        G(x^{N+1}) - G(x^*) \leq \frac{d^2(x^*,x^0)}{2\tau (N+1)}.
    \]
\end{proposition}

\begin{proof}
    Choose $y = x^*$ in \cref{key1} and use $\tau_k \geq \tau > 0$ to get
	\[
		G(x^{k+1}) - G(x^*) \leq \frac{1}{2\tau_k}(d^2(x^*,x^k)-d^2(x^*,x^{k+1})) \leq \frac{1}{2\tau}(d^2(x^*,x^k)-d^2(x^*,x^{k+1})).
	\]
    Summing over from $0$ to $N$ and using the fact that the sequence $\{G(x^k)\}_{k \in \N}$ is monotone nonincreasing (choose $y=x$ in \cref{key1} to get $G(x^N) \leq G(x^k)$ for all $k=0,1\ldots,N$) gives us the result
	\[
		(G(x^{N+1}) - G(x^*))(N+1) = \sum_{k=0}^{N}(G(x^{N+1}) - G(x^*)) \leq \sum_{k=0}^{N}(G(x^{k+1}) - G(x^*)) \leq \frac{d^2(x^*,x^0)}{2\tau}.
        \qedhere
	\]
\end{proof}

\subsection{Gradient descent map}

For a complete finite-dimensional Alexandrov space $(X,d)$ with $\partial X = \emptyset$ and curvature bounded from below by $\underline{\kappa} \in \R$, we can define the gradient descent map of a proper, convex, and lower semicontinuous function $F : X \to \overbar{\R}$. Fix the step size $\tau > 0$ and a closed, (geodesically) convex set $A \subset X$ containing a nonempty sublevel set of $F$.

\begin{definition}[Gradient descent map {\cite[Definition 4.5]{ohta2015discretetime}}]
	\label{gradresolv}
	For each $x \in X$ with $|\nabla_- F|(x) < \infty $, we define
	\[
		\gdesc_{\tau F} : X \to X, \qquad \gdesc_{\tau F}(x) \coloneqq \gexp(\tau \nabla(-F)(x)),
	\]
	where $\gexp$ is the gradient exponential map and $\nabla (-F)$ is as given in \cref{nabla}.
\end{definition}

\begin{remark}
    The rigorous definition of the gradient exponential map is very technical and we will not give it here as we will will not use it for our results.
    Informally, the gradient exponential map $\gexp : T_pX \to X$ maps a given $v \in T_pX$ to a gradient curve of the distance function from $p$, i.e., $\gexp(tv) = \gamma(t)$, where $\gamma : [0,\infty) \to X$ is a radial curve for $d_p(\cdot)$.
    These radial curves have a parametrisation that coincides with the arc length parametrisation as long as $\gamma$ is a geodesic, meaning that the gradient exponential agrees with the exponential map at all points where the latter is defined (the exponential map will be defined in \cref{sec:proxgrad}).
    The notable difference between the maps is that the gradient exponential map is defined for all times $t \geq 0$, whereas the exponential map (if defined at all) is defined only in the cut-locus (which can be arbitrary small). See \cite[Chapter 16 J]{petruninfoundations} for the definition and some of its properties.
\end{remark}

As in the upper curvature bound case, to prove convergence results we need the corresponding key lemma for the gradient descent map.

\begin{lemma}[{\cite[Lemma 4.6(II)]{ohta2015discretetime}}]
	\label{key2}
	Let $(X,d)$ be a finite-dimensional complete Alexandrov space with curvature bounded from below by $\underline{\kappa} < 0$, $\partial X = \emptyset$, and $\diam A < \infty$. Then we have
	\[
		d^2(y, \gdesc_{\tau F}(x)) \leq d^2(y,x) + 2\tau (F(y) - F(x)) + \frac{\lambda_l}{2} (\tau|\nabla_- F|(x))^2
	\]
    for all $x, y \in A$ satisfying $\gdesc_{\tau F}(x) \in A$, where $\lambda_l = \lambda_l(\underline{\kappa},\diam A)$ is the constant \cref{triglowlambda}.
\end{lemma}

As before, we consider the algorithm generated by the gradient descent map $\gdesc_{\tau F}$ by choosing an arbitrary starting point $x^0 \in A$ and iterating
\begin{equation}
	\label{desciter}
	x^{k+1} = \gdesc_{\tau_k F}(x^k), \quad k \geq 0,
\end{equation}
for an appropriate positive sequence $\{\tau_k \}_{k \in \N}$.
This algorithm is shown to converge to a minimiser of a Lipschitz continuous $F$ \cite[Theorem 5.5]{ohta2015discretetime} and to its unique minimiser when it is also strongly convex \cite[Theorem 5.7]{ohta2015discretetime}. These are done with the same step size assumptions as in the upper curvature bound case ($\tau_k > 0$, $\sum_{k=0}^{\infty}\tau_k = \infty$, and $\sum_{k=0}^{\infty}\tau_k^2 < \infty$). For these step sizes, we have the following iteration-complexity bound for the gradient descent algorithm.

\begin{proposition}
	\label{iter1}
	Let $F$ be Lipschitz continuous with constant $L_F$ and $\{x^k\}_{k \in \N}$ be the sequence generated by \cref{desciter}. Let $\alpha_k \coloneqq \tau_k |\nabla_-F|(x^k) $, for $k=0,1,\ldots$. Then for every $N \in$ $\N$, the following holds
	\[
		\min \{ F(x^k) - F(x^*) \; | \; k=0,1,\ldots,N \} \leq L_F \frac{d^2(x^0,x^*) + \frac{\lambda_l}{2}\sum_{k=0}^{N}\alpha_k^2}{2 \sum_{k=0}^{N} \alpha_k},
	\]
where $x^*$ is a point in which $F$ achieves a minimum.
\end{proposition}

\begin{proof}
Choose $y = x^*$ and insert $\tau_k = \alpha_k/|\nabla_-F|(x^k)$ in \cref{key2} to get
	\[
		2\frac{\alpha_k}{|\nabla_-F|(x^k)} (F(x^k) - F(x^*)) \leq d^2(x^k,x^*) - d^2(x^{k+1},x^*) + \frac{\lambda_l}{2} \alpha_k^2
	\]
for all $k=0,1,\ldots,N$. From the Lipschitz continuity of $F$ we have that $|\nabla_-F|(x^k) \leq L_F$, so summing over from $0$ to $N$ we get
	\[
		\frac{2}{L_F} \sum_{k=0}^{N} \alpha_k (F(x^k) - F(x^*)) \leq d^2(x^0,x^*) - d^2(x^{N+1},x^*) + \frac{\lambda_l}{2} \sum_{k=0}^{N} \alpha_k^2.
	\]
Simply noticing now that $\sum_{k=0}^{N} \alpha_k (F(x^k) - F(x^*)) \geq \min \{ F(x^k) - F(x^*) \; | \; k=0,1,\ldots,N \} \sum_{k=0}^{N} \alpha_k$ gives us
	\[
		\min \{ F(x^k) - F(x^*) \; | \; k=0,1,\ldots,N \} \sum_{k=0}^{N} \alpha_k \leq L_F \frac{d^2(x^0,x^*) + \frac{\lambda_l}{2} \sum_{k=0}^{N} \alpha_k^2}{2}
	\]
which when divided by the sum $\sum_{k=0}^{N} \alpha_k$ gives us the result.
\end{proof}

\section{Proximal gradient map}
\label{sec:proxgrad}

Now we are ready to define and analyse the proximal gradient map. Let $(X,d)$ be a complete locally compact Alexandrov space with $\partial X = \emptyset$ and curvature bounded from above and below by $\overline{\kappa}, \underline{\kappa} \in \R$, respectively ($\underline{\kappa} \leq \overline{\kappa}$). Fix the step size $\tau > 0$ and let $G, F : X \to \overbar{\R}$ be proper, convex, and lower semicontinuous functions. The proximal gradient map will use a different version of the gradient map relying on the exponential map instead of the gradient exponential map.

\begin{definition}[Exponential map]
    \label{def:exp}
	The exponential map at $x$ is a mapping from a subset of $T_xX$ into $X$ defined by $\exp_x (tv) = \gamma(t)$, where $\gamma : [0,\delta] \to X$ is a geodesic, with $\gamma(0) = x$ and $\gamma^+(0) = v$.
\end{definition}

Geometrically, the exponential map takes a geodesic direction $v$ at a starting point $x$ and travels to that direction for distance $t$ arriving at the point $\gamma(t)$. Clearly this is well-defined for at least sufficiently short vectors (as long as the curve $\gamma$ stays minimal). In fact, as $X$ is a bilaterally bounded Alexandrov space, it is a topological manifold and hence has infinitely extendible (to both sides) local geodesics by \cite[Exercise 9.9]{petruninfoundations}. This, along with local compactness, implies that the space of directions $\Sigma^{\prime}_{x}X$ (and hence $T^{\prime}_xX$ as well) is complete for any $x \in X$ \cite[Exercise 9.10]{petruninfoundations}. Thus, if we allow our geodesics to be only locally minimising, we see that the exponential map is defined for any vector in $T_xX$. This will be used in the a priori locality \cref{locality} but will subsequently be strengthened to geodesics that lie in a geodesically convex set. The exponential map also enjoys Lipschitz continuity.

\begin{lemma}[{\cite[2.3]{SobolevKuwae2001}}]
	\label{explip}
	The exponential map is Lipschitz continuous with a constant $L_{\exp} \coloneqq 1 + L(\overline{\kappa}, r)$, where $L(\overline{\kappa}, r) \geq 0$ and $r$ is a radius of the origin where the exponential map is defined, with $L(\overline{\kappa}, r) \to 0$ as $r \to 0$ (and naturally as $\overline{\kappa} \to 0$).
\end{lemma}

We can now reformulate the gradient descent map.

\begin{definition}[Gradient descent map, bilateral bounds]
	\label{gradexpresolv}
	For each $x \in X$ with $|\nabla_- F|(x) < \infty $, we define
	\[
		\desc_{\tau F} : X \to X, \qquad \desc_{\tau F}(x) \coloneqq \exp_x(\tau \nabla(-F)(x)),
	\]
    where $\exp$ is the exponential map and $\nabla (-F)$ is as given in \cref{nabla}.
\end{definition}

Combining our newly defined gradient descent map with the standard unconstrained proximal map ($A=X$ in \cref{def:proxconstrained} without the diameter constraint) gives us the proximal gradient map.

\begin{definition}[Proximal gradient map]
    \label{proxgrad}
    For each $x\in X$ with $|\nabla_- F|(x) < \infty $, we choose
    \[
        J : X \to X, \qquad J(x) \coloneqq \prox_{\tau G}(\desc_{\tau F}(x)) \in \argmin_{y \in X} \bigg\{G(y) + \frac{1}{2\tau}d^2(\exp_x(\tau \nabla (-F)(x)),y) \bigg\}.
    \]
\end{definition}

It can be shown for a proper, lower semicontinuous, and convex function $G$, that the proximal gradient map generates at least one point $J(x)$ for any $x \in X$ (this follows from the coercivity of the expression inside the braces). To analyse the map, we need a result akin to \cref{key1,key2}. To do so, we need an optimality condition, refine \cref{key1} to include the strong convexity factor of the squared distance function, modify \cref{key2} to include the squared distance function (replacing the squared vector term), and use a smoothness condition for $F$ to move from $F(x)$ to $F(J(x))$.

We can provide a necessary condition characterisation of the point $J(x)$ similarly to the basic proximal map in \cite[(4.2)]{ohta2015discretetime}.

\begin{lemma}
    \label{lemma:proxopt}
	Let $(X,d)$ be a locally compact Alexandrov space with an upper curvature bound $\overline{\kappa}$ and $G : X \to \overbar{\R}$ proper, lower semicontinuous, and $\lambda$-convex with $\lambda \geq 0$. Then for all unit-speed geodesics $\gamma, \eta : [0,1] \to X$  with $\gamma(0) = \prox_{\tau G}(x)$ and $\eta(t) = \prox_{\tau G}(x) \#_t x$ such that $0 < d(\prox_{\tau G}(x),x) < \pi/(2\sqrt{\overline{\kappa}})$, we have
    \begin{equation}
		\label{eq:proxopt}
        D_{\prox_{\tau G}(x)}G(\gamma^+(0)) - \frac{1}{\tau} \langle \gamma^+(0), \eta^+(0) \rangle \geq 0.
	\end{equation}
\end{lemma}

\begin{proof}
    By \cref{proxgrad} and the Fermat principle of \cref{fermat} (the sum in the proximal map is strongly convex by assumption $\lambda \geq 0$ and \cref{dsquared}), $\prox_{\tau G}(x)$ solves the minimisation problem for the proximal map if and only if $D_{\prox_{\tau G}(x)} (G + (2\tau)^{-1}d_x^2)( \gamma^+(0)) \geq 0$ for all $\gamma^+(0) \in T_{\prox_{\tau G}(x)}X$.
    The first variation formula of \cref{firstvar} then shows that $D_{\prox_{\tau G}(x)}(2\tau)^{-1}d_x^2(\gamma^+(0)) = -\tau^{-1} \langle \gamma^+(0), \eta^+(0) \rangle$ for all unit-speed geodesics $\eta : [0,1] \to X$ with $\eta(t) = \prox_{\tau G}(x) \#_t x$.
\end{proof}

To refine \cref{key1}, we use \cref{trigupper}, the optimality condition of the proximal map, and the convexity of $G$ to get the appropriate estimate.
The following few lemmas are a posteriori estimates that will require first showing that $\prox_{\tau G}(x), \desc_{\tau F}(x) \in A$.

\begin{lemma}
	\label{key1new}
    Let $(X,d)$ be a complete Alexandrov space with curvature bounded from above by $\overline{\kappa}$ and $G : X \to \overbar{\R}$ proper, lower semicontinuous, and convex. Let $A \subset X$ be such that $\diam A < \pi/(2\sqrt{\overline{\kappa}})$. Then we have
    \begin{equation}
		\label{eq:key1new}
        d^2(y, \prox_{\tau G}(x)) \leq d^2(y,x) + 2\tau (G(y) - G(\prox_{\tau G}(x))) - \frac{\lambda_u}{2}d^2(x, \prox_{\tau G}(x)),
	\end{equation}
    for all $x, y \in A$ satisfying $\prox_{\tau G}(x) \in A$, where $\lambda_u = \lambda_u(\overline{\kappa},\diam A)$ is given in \cref{triguplambda}.
\end{lemma}

\begin{proof}
	Choose $a = d(y,x)$, $b = d(x,\prox_{\tau G}(x))$, and $c = d(y,\prox_{\tau G}(x))$ in \cref{trigupper} to get
	\begin{equation}
        \label{eq:key1new:item0}
        \begin{split}
            d^2(y,x) &\geq \frac{\sqrt{\overline{\kappa}}c}{\tan(\sqrt{\overline{\kappa}}c)}d^2(x, \prox_{\tau G}(x)) + d^2(y, \prox_{\tau G}(x)) \\
            &- 2d(x, \prox_{\tau G}(x))d(y, \prox_{\tau G}(x)) \cos (\theta) \\
            &\geq \frac{\lambda_u}{2}d^2(x, \prox_{\tau G}(x)) + d^2(y, \prox_{\tau G}(x)) - 2d(x, \prox_{\tau G}(x))d(y, \prox_{\tau G}(x)) \cos (\theta),
        \end{split}
	\end{equation}
	where $\theta = \angle_{\prox_{\tau G}(x)} (\gamma,\eta)$ is the angle between the sides $b$ and $c$, and $\gamma(t) = \prox_{\tau G}(x) \#_t y$ and $\eta(t) = \prox_{\tau G}(x) \#_t x$. The definition of the “inner product”, followed by \cref{lemma:proxopt} and convexity \cref{subdiff} of $G$ along the geodesic $\gamma$, establish
	\begin{equation}
        \label{eq:key1new:item1}
        \begin{split}
        - 2d(x, \prox_{\tau G}(x))d(y, \prox_{\tau G}(x)) \cos (\theta)
        &
        = -2 \iprod{\gamma^+(0)}{\eta^+(0)}
        \\
        &
        \ge
        -2\tau D_{\prox_{\tau G}(x)}G(\gamma^+(0))
        \\
        &
        \ge 2\tau (G(\prox_{\tau G}(x)) - G(y)).
        \end{split}
	\end{equation}
    Combining \cref{eq:key1new:item0,eq:key1new:item1} establishes the claim.
\end{proof}

To modify \cref{key2}, we use \cref{triglower}, the optimality condition of the gradient map \cref{optimalitygrad}, and convexity of $F$ to get the appropriate estimate.

\begin{lemma}
	\label{key2new}
    Let $(X,d)$ be a complete Alexandrov space with curvature bounded from above by $\overline{\kappa}$ and below by $\underline{\kappa}$, and $F : X \to \overbar{\R}$ proper, lower semicontinuous, and convex. Let $A \subset X$ be such that $\diam A < \pi/(2\sqrt{\overline{\kappa}})$ and $0 < |\nabla_-F|(x) < \infty$, for all $x \in A$. Then we have
    \begin{equation}
		\label{eq:key2new}
        d^2(y, \desc_{\tau F}(x)) \leq d^2(y,x) + 2\tau (F(y) - F(x)) + \frac{\lambda_l}{2}d^2(x, \desc_{\tau F}(x)),
	\end{equation}
    for all $x, y \in A$ satisfying $\desc_{\tau F}(x) \in A$, where $\lambda_l = \lambda_l(\underline{\kappa},\diam A)$ is the constant \cref{triglowlambda}.
\end{lemma}

\begin{proof}
    Since $A$ is assumed geodesically convex and $x, y, \desc_{\tau F}(x) \in A$, we have by \cref{triglower}
	\[
		d^2(y, \desc_{\tau F}(x)) \leq d^2(y,x) + \frac{\lambda_l}{2}d^2(x, \desc_{\tau F}(x)) - 2d(x,y)d(x,\desc_{\tau F}(x))\cos(\theta).
	\]
    It follows from $\desc_{\tau F}(x) = \exp_x(\tau \nabla(-F)(x))$ and the definition of the “inner product” for geodesics $\eta(t) = x \#_t \exp_x (\tau \nabla(-F)(x))$ and $\gamma(t) = x \#_t y$ that
    \[
        - 2d(x,y)d(x,\desc_{\tau F}(x))\cos(\theta) = -2 \langle \eta^+(0), \gamma^+(0) \rangle = -2 \langle \tau \nabla(-F)(x), \gamma^+(0) \rangle.
    \]
    Furthermore, by convexity along $\gamma$ and the characterisation \cref{optimalitygrad} we have
	\begin{equation}
		\label{eq:key2new:item1}
		F(y) - F(x) \geq D_x F(\gamma^+(0)) \geq -\langle \nabla(-F)(x), \gamma^+(0) \rangle.
	\end{equation}
    The claim follows by combining the estimates.
\end{proof}

Finally, we need to introduce the concept of uniform smoothness of a function, which is a natural reformulation of the same concept in Banach spaces and a necessary assumption in analysing the proximal gradient map.

\begin{definition}[Uniform smoothness]
    A function $F : X \to \overline{\R}$ is called uniformly smooth with factor $L>0$, if for all $x, y \in X$ and geodesics $\gamma(t) = x \#_t y$ with $t \in [0,1]$, we have
	\[
		F(\gamma(t)) + \frac{L}{2}t(1-t)d^2(x,y) \geq (1-t)F(x) + tF(y).
	\]
\end{definition}

\begin{example}
	Let $X$ be an Alexandrov space with curvature bounded from below. The beginning of \cref{subsec:trig} shows that the squared distance function $d^2_z$, for a fixed $z \in X$, is $-\lambda$-smooth, where $\lambda$ is the factor from \cref{minusdsquared}.
\end{example}

Using uniform smoothness to derive a suitable descent lemma, we need to also assume differentiability from $F$. This assumption is justified as our space $X$ is bilaterally bounded and thus contains the structure of a differentiable manifold by \cref{bilat}. For a differentiable $F$ , we have the correspondence
\[
    \nabla(-F)(x) = -\nabla F(x) \quad \text{such that} \quad D_xF(\gamma^+(0)) = \langle \nabla F(x), \gamma^+(0) \rangle,
\]
for all geodesics $\gamma$ with $\gamma(0)=x$ (see, \cite[Section 4.1]{ohtawasserstein} and \cite[Chapter 1.3]{ambrosio2006gradient}, respectively). We call $\nabla F(x)$ the (standard) gradient of $F$ at $x$. Note that for critical points $x^*$ of $F$ we have $\nabla F(x^*) = 0$.

\begin{lemma}[Descent lemma]
	\label{descent}
	Let $(X,d)$ be a complete locally compact Alexandrov space and $F$ be proper, $\lambda$-convex, lower semicontinuous, and uniformly smooth with factor $L>0$. We have (for any instance of the possibly non-unique $\log_x y$ when $(X, d)$ is not upper-bounded)
	\begin{equation}
	    \label{eq:descent:general}
		D_xF(\log_x y) + \frac{L}{2}d^2(x,y) \geq F(y) - F(x).
	\end{equation}
	If, $(X, d)$ is, moreover, bilaterally bounded, and $F$ differentiable (but not necessarily $\lambda$-convex), then, for $\gamma(t)=x \#_t y$, we have
	\[
		-\iprod{\gamma^+(0)}{\nabla(-F)(x)}  + \frac{L}{2}d^2(x,y) + F(x) \geq F(y).
	\]
\end{lemma}

\begin{proof}
    By the uniform smoothness we have
	\[
		\frac{F(\gamma(t)) - F(x)}{t} + \frac{L}{2}(1-t)d^2(x,y) \geq F(y) - F(x),
	\]
    which letting $t \downto 0$ yields
	\[
		D_xF(\gamma^+(0)) + \frac{L}{2}d^2(x,y) \geq F(y) - F(x).
	\]
	That is, \eqref{eq:descent:general} holds. Lastly, the bilateral bounds and the differentiability of $F$ give $D_xF(\gamma^+(0)) = \langle \gamma^+(0), \nabla F(x) \rangle = -\langle \gamma^+(0), \nabla(-F)(x) \rangle$.
\end{proof}

\subsection{The key estimate}

We now have the necessary tools to derive our key inequality for the proximal gradient map. The plan is to combine \cref{key1new,key2new} at specific points, using the uniform smoothness of $F$ to move from the estimates on $F$ from being based at $x$ to $J(x)$ (to be able combine with the estimate of $G$ at $J(x)$ in \cref{key1new}). Note that it is an a posteriori result, which still requires showing that $w, J(x) \in A$.

\begin{lemma}[Key lemma, convex $G$ and $F$]
    \label{keythm}
    Let $(X,d)$ be a complete locally compact Alexandrov space with $\partial X = \emptyset$ and curvature bounded from above by $\overline{\kappa} \in \R$ and below by $\underline{\kappa} \in \R$ ($\underline{\kappa} \leq \overline{\kappa}$). Fix the step size $\tau > 0$ and a closed, geodesically convex set $A \subset X$ ($\diam A < \pi/(2\sqrt {\overline{\kappa}})$, if $\overline{\kappa} > 0$) containing a nonempty sublevel set of $G + F$, where $G, F : X \to \overbar{\R}$ are proper, convex, and lower semicontinuous functions. Assume further that $F$ is differentiable and uniformly smooth with factor $L>0$. Then we have
    \begin{equation}
        \label{key}
        \begin{split}
        d^2(y,J(x)) \leq d^2(y,x) &+ 2\tau (H(y) - H(J(x))) + (L\tau - \frac{\lambda_u}{2})d^2(x,J(x)) \\
        &+ \left(\frac{\lambda_l}{2} - 1\right)d^2(x,w) + \left(1 - \frac{\lambda_u}{2}\right)d^2(w, J(x))
        \end{split}
    \end{equation}
    for all $x, y \in A$ satisfying $w, J(x) \in A$, where $H(x) = G(x) + F(x)$, $w = \exp_x(\tau \nabla(-F)(x))$, along with
    \[
        \frac{\lambda_u}{2} = \frac{\sqrt{\overline{\kappa}}\diam A}{\tan(\sqrt{\overline{\kappa}}\diam A)} \quad \text{and} \quad \frac{\lambda_l}{2} = \frac{\sqrt{-\underline{\kappa}}\diam A}{\tanh(\sqrt{-\underline{\kappa}}\diam A)}.
    \]
\end{lemma}

Much like the corresponding \cref{key1,key2} for the proximal and gradient descent maps, respectively, this estimate will be used to prove convergence results related to the iteration scheme defined by the proximal gradient map.

\begin{proof}
	By \cref{key2new} we have
	\begin{equation}
	\label{1}
		d^2(y, w) \leq d^2(y,x) + 2\tau (F(y) - F(x)) + \frac{\lambda_l}{2} d^2(x, w),
	\end{equation}
	where $w = \desc_{\tau F}(x) = \exp_x(\tau \nabla(-F)(x))$. Assuming that $F$ is uniformly smooth with factor $L>0$, we have by \cref{descent} that
	\begin{equation}
	\label{2}
		-2\tau F(x) \leq -2\tau F(J(x)) + L\tau d^2(x,J(x)) - 2\langle \gamma^+(0), \tau \nabla(-F)(x) \rangle
	\end{equation}
	for a geodesic $\gamma(t)=x \#_t J(x)$.
	By \cref{trigupper} we have
	\begin{equation}
	\label{3}
	\begin{split}
		- 2\langle \gamma^+(0), \tau \nabla(-F)(x) \rangle &= -2d(x,J(x))d(x,w) \cos (\theta) \\
		&\leq d^2(w,J(x)) - \frac{\lambda_u}{2} d^2(x,J(x)) - d^2(x,w),
	\end{split}
	\end{equation}
	where $\theta$ is the angle between geodesics from $x$ to $J(x)$ and $w$. Combining expressions \cref{1,2,3} gives us
	\begin{equation}
	\label{4}
	\begin{split}
		d^2(y, w) \leq d^2(y,x) &+ 2\tau (F(y) - F(J(x))) + (L\tau - \frac{\lambda_u}{2}) d^2(x,J(x)) \\
		&+ \left(\frac{\lambda_l}{2}-1\right) d^2(x,w) + d^2(w,J(x)).
	\end{split}
	\end{equation}
	Finally, let $x=w$ in \cref{key1new} to obtain
	\begin{equation}
	\label{5}
		d^2(y, J(x)) \leq d^2(y,w) + 2\tau (G(y) - G(J(x))) - \frac{\lambda_u}{2}d^2(w, J(x)),
	\end{equation}
	and sum estimates \cref{4,5} to obtain
	\begin{equation*}
	\begin{split}
		d^2(y,J(x)) \leq d^2(y,x) &+ 2\tau (H(y) - H(J(x))) + (L\tau - \frac{\lambda_u}{2})d^2(x,J(x)) \\
		&+ \left(\frac{\lambda_l}{2}-1\right)d^2(x,w) + \left(1 - \frac{\lambda_u}{2}\right)d^2(w, J(x)),
	\end{split}
	\end{equation*}
	where $H(x) = G(x) + F(x)$.
\end{proof}

\section{Forward-backward method}
\label{sec:fb}

We can now consider the algorithm generated by the proximal gradient map $J$ of \cref{proxgrad}. Fixing an arbitrary starting point $x^0 \in X$, the forward-backward method is defined as the iterative process
\begin{equation}
	\label{proxgraditer}
	x^{k+1} = J(x^k), \quad k \geq 0
\end{equation}
for an appropriately chosen positive sequence $\{ \tau_k \}_{k \in \N}$.
We generally denote the intermediate step by $w^k = \desc_{\tau F}(x^k) = \exp_{x^k}(\tau_k \nabla(-F)(x^k))$, so that $x^{k+1}=\prox_{\tau_k G}(w^k)$.

Before showing that method \cref{proxgraditer} converges to a minimum of the function $H = F + G$, let us do some preliminary analysis for cases of specific curvature bounds.  Assume that we can use \cref{keythm} for all $k \geq 0$.

\begin{itemize}
    \item
    $\underline{\kappa} = 0 = \overline{\kappa}$ (flat case: e.g. $\R^n$ or a Hilbert space): now $\lambda_u/2=1=\lambda_l/2$, giving us the standard estimate
    \begin{equation*}
        d^2(y,x^{k+1}) \leq d^2(y,x^k) + 2\tau_k (H(y) - H(x^{k+1})) + (L\tau_k - 1)d^2(x^k,x^{k+1})
    \end{equation*}
    for all $y$.
    Assuming $0 < L \tau_{\min} \leq L \tau_k < 1$ for all $k \geq 0$ establishes the Fej{\'e}r property of the sequence and leads to the usual convergence results.
    \item
    $\underline{\kappa} \leq 0 = \overline{\kappa}$ (e.g. a Hadamard space with an additional lower curvature bound): $\lambda_u/2=1$, giving us the estimate
    \begin{equation*}
    \begin{split}
        d^2(y,x^{k+1}) \leq d^2(y,x^k) &+ 2\tau_k (H(y) - H(x^{k+1})) \\
        &+ (L\tau_k - 1)d^2(x^k,x^{k+1}) + \left(\frac{\lambda_l}{2}-1\right)d^2(x^k,w^k).
    \end{split}
    \end{equation*}
    Since $\lambda_l/2 - 1 \geq 0$, we need to estimate the term $d^2(x^k,w^k)$, for which, by \cref{explip}, we get from Lipschitz continuity of the exponential map (see the dicussion after \cref{def:exp}) that $d^2(x^k,w^k) \leq (L_{\exp} \tau_k |\nabla_- F|(x^k))^2$. By further assuming that $F$ is, for example, Lipschitz continuous with factor $L_F$, we have $|\nabla_- F|(x^k) \leq L_F$ and $(L_{\exp} \tau_k |\nabla_- F|(x^k))^2 \leq (L_{\exp} \tau_k L_F)^2$. This shows that along with the previous assumption of $L \tau_k < 1$, we also need to assume $\sum_{k=0}^{\infty} \tau_k^2 < \infty$ to obtain convergence results.
    \item
    $\underline{\kappa} = 0 \leq \overline{\kappa}$: now $\lambda_l/2=1$, giving us the estimate
    \begin{equation*}
    \begin{split}
        d^2(y,x^{k+1}) \leq d^2(y,x^k) &+ 2\tau_k (H(y) - H(x^{k+1})) \\
        &+ (L\tau_k - \frac{\lambda_u}{2})d^2(x^k,x^{k+1}) + \left(1 - \frac{\lambda_u}{2}\right)d^2(w^k,x^{k+1}).
    \end{split}
    \end{equation*}
    Since $1 - \lambda_u/2 \geq 0$, we need to estimate the term $d^2(w^k,x^{k+1})$ as well as assume that $L \tau_k < \lambda_u/2$. One way of dealing with $d^2(w^k,x^{k+1})$ is to use the definition of the proximal map to get $G(x^{k+1}) + (2\tau_k)^{-1}d^2(w^k,x^{k+1}) \leq G(w^k)$. Assume further that $G$ is Lipschitz continuous with factor $L_G$ so that $d^2(w^k,x^{k+1}) \leq 2 \tau_k [G(w^k) - G(x^{k+1})] \leq 2 \tau_k L_Gd(w^k,x^{k+1})$ giving $d^2(w^k,x^{k+1}) \leq 4 \tau_k^2 L^2_G$. This again suggests that we need to require $\sum_{k=0}^{\infty} \tau_k^2 < \infty$. However, the Lipschitz assumption on $G$ is rather unpleasant as it excludes, for example, the indicator function.
    \item
    $\underline{\kappa} \leq \overline{\kappa}$: here we cannot simplify at all and simply have the original estimate \cref{key}. For convergence results, we should need all of the previous assumptions: $L \tau_k < \lambda_u/2$, $\sum_{k=0}^{\infty} \tau_k^2 < \infty$, and a bound for $d^2(w^k,x^{k+1})$.
\end{itemize}
It appears to be clear from these points, that in the presence of nonzero curvature, the rather strict assumption of square summability of the step sizes is essential in obtaining convergence results. This ultimately arises from the strong convexity and strong concavity factors of the squared distance function in spaces with upper and lower curvature bounds, respectively.

\subsection{Convergence analysis}

Before analysing convergence, we gather the fundamental assumptions regarding the forward-backward algorithm here.

\begin{assumption}
	\label{assum:space}
	The space $(X,d)$ is a complete locally compact bilaterally bounded Alexandrov space with $\partial X = \emptyset$ (from below by $\underline{\kappa} \in \R$ and above by $\overline{\kappa} \in \R$, $\underline{\kappa} \leq \overline{\kappa}$).
\end{assumption}

\begin{assumption}
	\label{assum:func}
	The functions $G, F : X \to \overline{\R}$ are proper, convex, and lower semicontinuous and $F$ is differentiable and uniformly smooth with factor $L>0$ (to be precise, these only need to apply in a closed geodesically convex set containing a nonempty sublevel set of $H = G + F$).
\end{assumption}

\begin{assumption}
	\label{assum:min}
	The point $x^* \in X$ is some global minimiser of $H = G + F$.
\end{assumption}

We first start by showing that the forward-backward method \cref{proxgraditer} creates a monotone sequence of function values of $H$ which in turn shows the convergence of the iterates.
This is still an a posteriori result, which requires first showing that $x^{k+1}, \desc_{\tau F}(x^k) \in A$.

\begin{proposition}[A posteriori monotonicity]
	\label{prop:proxgrad:monotonicity}
	Along with \cref{assum:space}, let $G$ be lower semicontinuous and $F$ differentiable, and uniformly smooth with factor $L>0$. Fix a closed, geodesically convex set $A \subset X$ containing a nonempty sublevel set of $G + F$ and, if $\overline{\kappa} > 0$, such that $\diam A < \pi/(2\sqrt {\overline{\kappa}})$.
    Let $\{ \tau_k \}_{k \in \N}$ be a positive sequence with $L \tau_k < \lambda_u/2$ for all $k \geq 0$.
    If the points $x^k$, $x^{k+1}$, and $w^k = \desc_{\tau F}(x^k)$ generated by the forward-backward method \cref{proxgraditer} belong to $A$, then
	\begin{equation}
        \label{eq:proxgrad:monotonicity}
		H(x^{k+1}) \leq H(x^k) + (\frac{L}{2} - \frac{\lambda_u}{4 \tau_k})d^2(x^k, x^{k+1}).
	\end{equation}
    In particular, the sequence $\{H(x^k)\}_{k \in \N}$ is monotonically nonincreasing. Furthermore, if $x^0 \in$ \emph{dom} $H$ and $\inf H > -\infty$, then $d(x^k,x^{k+1}) \to 0$ as $k \to \infty$.
\end{proposition}

\begin{proof}
    \Cref{proxgrad} and the iteration scheme \cref{proxgraditer} give
    \begin{equation}
        \label{eq:proxgrad:monotonicity:0}
        G(x^{k+1}) + \frac{1}{2 \tau_k}d^2(w^k, x^{k+1}) \leq G(x^k) + \frac{1}{2 \tau_k}d^2(w^k, x^k),
    \end{equation}
    where $w^k = \desc_{\tau_k F}(x^k) = \exp_{x^k}(\tau_k \nabla(-F)(x^k))$. By \cref{eq:proxgrad:monotonicity:0} and \cref{descent} we have
	\begin{equation}
        \label{eq:proxgrad:monotonicity:1}
        \begin{split}
            H(x^{k+1}) &= G(x^{k+1}) + F(x^{k+1}) \\
            &\leq H(x^k) + \frac{L}{2}d^2(x^k, x^{k+1}) + \frac{1}{2 \tau_k} \bigg( d^2(w^k, x^k) - d^2(w^k, x^{k+1}) \bigg) - \langle \gamma^+(0), \nabla(-F)(x^k) \rangle,
        \end{split}
	\end{equation}
    for a geodesic $\gamma(t) = x^k \#_t x^{k+1}$. Now use \cref{trigupper}, \eqref{eq:pythagoras-replacement} (i.e., our replacement of the Pythagoras' identity in Hilbert spaces) and \cref{innprod} to estimate
	\begin{equation}
        \label{eq:proxgrad:monotonicity:2}
        \begin{split}
            d^2(w^k, x^k) - d^2(w^k, x^{k+1}) &\leq -\frac{\lambda_u}{2} d^2(x^k, x^{k+1}) + 2d(x^k, x^{k+1})d(x^k, w^k) \cos (\theta) \\
            &= -\frac{\lambda_u}{2} d^2(x^k, x^{k+1}) + 2 \langle \gamma^+(0), \tau_k \nabla (-F)(x^k) \rangle.
        \end{split}
	\end{equation}
    Putting \cref{eq:proxgrad:monotonicity:1,eq:proxgrad:monotonicity:2} together gives \cref{eq:proxgrad:monotonicity}.
    The latter implies that the sequence $\{H(x^k)\}_{k \in \N}$ is monotonically nonincreasing when $L \tau_k < \lambda_u/2$ for all $k \geq 0$.

    Now let $\epsilon$ satisfy $\lambda_u/(4 \tau_k) - L/2 \ge \epsilon > 0$ so that \cref{eq:proxgrad:monotonicity} yields
	\[
		\epsilon d^2(x^k, x^{k+1}) \leq (\frac{\lambda_u}{4 \tau_k} - \frac{L}{2})d^2(x^k, x^{k+1}) \leq H(x^k) - H(x^{k+1}).
	\]
    Let $x^0 \in$ dom $H$ and sum over from $0$ to $N$ to get
	\[
		\epsilon \sum_{k=0}^N d^2(x^k, x^{k+1}) \leq \sum_{k=0}^N (H(x^k) - H(x^{k+1})) = H(x^0) - H(x^{N+1}) \leq H(x^0) - \inf H < \infty.
	\]
    It follows that $d(x^k,x^{k+1}) \to 0$ since $\epsilon > 0$.
\end{proof}

As \cref{keythm,prop:proxgrad:monotonicity} need the iterates to stay in the same neighborhood as the reference points, we need to perform some locality analysis to determine when this is true, i.e., how close do we need to start from the reference point and how small do the step sizes need to be for the statement to hold.
The next a priori lemma proves that $x^{k+1}, w^k = \desc_{\tau F}(x^k) \in A$ as required by the former a posteriori results. Note also that here we allow the geodesics from $x^k$ to $w^k$ to be just locally minimising. This discrepancy will be resolved in the convergence proofs with an appropriate diameter bound for $A$.

\begin{lemma}[A priori locality]
	\label{locality}
	Along with \cref{assum:space,assum:min}, let $G, F : X \to \overbar{\R}$, where $G$ is proper and $F$ is differentiable and Lipschitz continuous with factor $L_F$. Let $\{x^k\}_{k \in \N}$ be generated by \cref{proxgraditer} and $x^k \in \widebar{B}(x^*, r)$ for some $r > 0$. If $R > 2r$ and
    \begin{equation}
        \label{localstep}
        0 < \tau_k \leq \frac{R - 2r}{2L_F(1 + L_{\exp})},
    \end{equation}
    then $w^k, x^{k+1} \in \widebar{B}(x^*, R)$.
\end{lemma}

\begin{proof}
    The definition of the proximal map gives
	\begin{equation}
		\label{eq:localstep:prox}
		G(x^{k+1}) + \frac{1}{2\tau_k}d^2(x^{k+1},w^k) \leq G(y) + \frac{1}{2\tau_k}d^2(y,w^k)
        \quad\text{for all}\quad
        y \in X,
	\end{equation}
    where $x^{k+1} = \prox_{\tau_k G}(w^k)$ and $w^k = \exp_{x^k}(\tau_k \nabla(-F)(x^k))$. Since $G(y) + F(y) \leq G(x^{k+1}) + F(x^{k+1})$ for some global minimising point $y = x^*$, then immediately
	\[
		d^2(x^{k+1},w^k) \leq 2 \tau_k (F(x^{k+1})-F(x^*)) + d^2(x^*,w^k) \leq 2\tau_k L_F d(x^{k+1},x^*) + d^2(x^*,w^k)
	\]
    by the Lipschitz continuity of $F$. By this and the triangle inequality
	\[
		d(x^{k+1},x^*) \leq d(x^{k+1},w^k) + d(x^*,w^k) \leq \sqrt{2\tau_k L_F d(x^{k+1},x^*) + d^2(x^*,w^k)} + d(x^*,w^k).
	\]
    Thus, rearranged and squared
	\[
        \begin{split}
            d^2(x^{k+1},x^*)+d^2(x^*,w^k) - 2d(x^{k+1},x^*)d(x^*,w^k) &=
            (d(x^{k+1},x^*)-d(x^*,w^k))^2 \\
            &\le
            2\tau_k L_F d(x^{k+1},x^*) + d^2(x^*,w^k).
        \end{split}
	\]
	That is,
	\[
		d^2(x^{k+1},x^*) - 2d(x^{k+1},x^*)d(x^*,w^k)
		\le
		2\tau_k L_F d(x^{k+1},x^*),
	\]
	which divided by $d(x^{k+1},x^*)$ when this quantity is non-zero, gives
	\[
		d(x^{k+1},x^*)
		\le
		2\tau_k L_F
		+ 2d(x^*,w^k).
	\]
    Obviously, this inequality also holds when $d(x^{k+1},x^*)=0$.
    By the triangle inequality, \cref{explip}, and the Lipschitz continuity of $F$ and of the exponential map (see the dicussion after \cref{def:exp}), we have
    \begin{equation}
        \label{localwbound}
        \begin{aligned}[t]
            d(x^*,w^k) &\leq d(x^*,x^k) + d(x^k,\exp_{x^k} (\tau_k \nabla (-F)(x^k))) \\
            &\leq d(x^*,x^k) + L_{\exp}\tau_k |\nabla_-F|(x^k) \\
            &\leq d(x^*,x^k) + L_F L_{\exp}\tau_k.
        \end{aligned}
    \end{equation}
    Thus,
    \begin{equation}
        \label{localbound}
        d(x^{k+1},x^*) \leq 2d(x^*,x^k) + 2L_F (1 + L_{\exp})\tau_k.
    \end{equation}
    Now, using \cref{localstep} and $x^k \in \widebar{B}(x^*, r)$ in \cref{localwbound} and \cref{localbound}, respectively, we obtain
    \[
        d(x^*,w^k) \leq r + L_F L_{\exp}\frac{R - 2r}{2L_F(1 + L_{\exp})} \leq r + L_{\exp}\frac{R - 2r}{2L_{\exp}} = \frac{R}{2} < R,
    \]
    as well as
    \[
        d(x^{k+1},x^*) \leq 2r + 2L_F(1 + L_{\exp})\frac{R - 2r}{2L_F(1 + L_{\exp})} = R.
    \]
    In other words, $x^{k+1}, w^k \in \widebar{B}(x^*, R)$.
\end{proof}

\begin{remark}[Combined step with lower bound only]
    \label{rem:total-surrogate}
    For any instance of the possibly non-unique (when $(X, d)$ is not upper-bounded) $\log_{\thisx} x$, let us define $\nextx$ as a minimiser of the \term{total surrogate}
    \[
        \phi_k(x) \defeq G(x) + F(\thisx) + D_{\thisx}F(\log_{\thisx} x) + \frac{1}{2\tau_k} d^2(\thisx, x);
    \]
    such methods have been studied in the Riemannian setting in \cite{stiefelproxgrad,riemproxgradhuang}. In a Hilbert space, this step is equivalent to \eqref{proxgraditer}.
    In a general Alexandrov space, we may not have equivalency.
    Nevertheless, we have by definition that $\phi_k(\nextx) \le \phi_k(\thisx)$, which combined with the descent inequality \eqref{eq:descent:general} for a uniformly smooth $F$ yields
    \[
        [F+G](\nextx) + \frac{\inv\tau_k - L}{2} d^2(\thisx, \nextx) \le [F+G](\thisx).
    \]
    Thus, by standard technique, if $F+G > -\infty$ and $\sup_k \tau_k L < 1$, we obtain $\sum_{k=0}^\infty d^2(\thisx, \nextx) <\infty$, hence $d^2(\thisx, \nextx) \to 0$.
    In a Hilbert space, this property readily implies convergence of subdifferentials to zero; see, e.g, \cite[Theorem 4.19]{tuomov2024tracking}.
    In a general Alexandrov space, such a characterisation appears more challenging to obtain, however, assuming for a subsequence that both $x^{k_\ell}$ and $x^{k_\ell+1}$ are, e.g., on the same polyhedron of a polyhedral complex, we can use the Hilbert space characterisation to obtain subdifferential convergence for the subsequence.
    Through the equivalency to \eqref{proxgraditer} in this case, this holds even for our proposed method with split steps.
\end{remark}

When $G$ and $F$ are assumed to be Lipschitz continuous, the convergence of the forward-backward method to a minimiser of $H$ is an application of \cref{locality} and the following deterministic variant of the Siegmund--Robbins lemma.

\begin{lemma}[\cite{robbins1971convergence}]
	\label{quasivariant}
	Let $\{ a_k \}_{k \in \N}$, $\{ b_k \}_{k \in \N}$ and $\{ c_k \}_{k \in \N}$ be sequences of nonnegative real numbers such that $a_{k+1} \leq a_k - b_k + c_k$ for each $k \geq 0$ and $\sum^{\infty}_{k=0} c_k < \infty$. Then the sequence $\{ a_k \}_{k \in \N}$ converges and $\sum^{\infty}_{k=0} b_k < \infty$.
\end{lemma}

\begin{theorem}[Convergence to a minimiser, Lipschitz continuous $G$ and $F$]
	\label{lip}
	Along with \cref{assum:space,assum:func,assum:min}, let $G$ and $F$ be Lipschitz continuous with factors $L_G$ and $L_F$, respectively. Let $x^0 \in \widebar{B}(x^*,r)$ for $r>0$ and $R>2r$ such that $R < \pi/(4\sqrt{\overline{\kappa}})$. Let $\{ \tau_k \}_{k \in \N}$ be a positive sequence of step sizes satisfying $\sum_{k=0}^{\infty} \tau_k = \infty$, $\sum_{k=0}^{\infty} \tau_k^2 < \infty$, and
	\begin{equation*}
		\tau_k \leq \min \bigg\{\frac{\lambda_u}{2L}, \frac{R - 2r}{2L_F(1 + L_{\exp})} \bigg\}
	\end{equation*}
	for all $k \geq 0$.
    If $\inf_{x \in X} H(x)$ is attained at some point, then the sequence $\{ x^k \}_{k \in \N}$ created by the iterative process \cref{proxgraditer} converges to some global minimiser of $H$ in $\widebar{B}(x^*,R)$ as $k \to \infty$.
\end{theorem}

\begin{proof}
    Set $A \defeq \widebar{B}(x^*,R)$ so that $\diam A \le 2R < \pi/(2\sqrt{\overline{\kappa}})$.
    Let $c \defeq \sum_{k=0}^\infty c_k$, as well as $R \defeq \sqrt{r^2 + c}$ for
    \[
		c_k = \bigg(\left(\frac{\lambda_l}{2}-1\right)L_{\exp}^2L_F^2 + \left(1 - \frac{\lambda_u}{2}\right)4L^2_G\bigg)\tau_k^2.
	\]
    We will, along the course of the proof, establish by induction that $x^k \in \widebar{B}(x^*,R)$ for all $k \in \N$.
    By assumption, we have $x^0 \in \widebar{B}(x^*, r) \subset A$, taking care of the inductive basis.
    So, for the inductive step, suppose $\{x^j\}_{j=0}^k \subset \widebar{B}(x^*,r)$.
    By \cref{locality}, we have $\nextx \in A$, so we may apply \cref{keythm}.
	The Lipschitz continuity of $G$ and the definition of the proximal map establish
    \[
        \frac{1}{2\tau_k}d^2(\nextx, \this w) \le G(\this w) - G(\nextx) \le L_G d(\nextx, \this w),
    \]
    hence $d^2(w^k,x^{k+1}) \leq 4 \tau_k^2 L^2_G$.
    Likewise, the Lipschitz continuity of $F$ and the exponential map yield $d^2(x^k,w^k) \leq L_{\exp}^2 \tau_k^2 L_F^2$ (see the examples from the beginning of \cref{sec:fb}).
    Using these estimates in the claim \eqref{key} of \cref{keythm} establishes
	\begin{equation}
        \label{eq:lip:descent}
        d^2(x^*,x^{k+1}) \leq d^2(x^*,x^k) - 2\tau_k (H(x^{k+1}) - H(x^*)) + c_k
	\end{equation}
    for all $k \in \N$.
    Through our inductive argument, we thus have $d^2(x^*,x^{j+1}) \leq d^2(x^*,x^j) + c_j$ for all $j=0,\ldots,k$.
    Summing, this establishes $\nextx \in \widebar{B}(x^*,R)$, completing the inductive step.

    Thus $x^k \in \widebar{B}(x^*,R)$ and \eqref{eq:lip:descent} for all $k \in \N$.
    Set $a_k \coloneqq d^2(x^*,x^k)$, $b_k \coloneqq 2\tau_k (H(x^{k+1}) - H(x^*))$.
    The sequence $\{ b_k \}_{k \in \N}$ is positive due to $H(x^*) \leq H(x^{k+1})$ and $\tau_k > 0$, whereas the positivity of $\{ c_k \}_{k \in \N}$ follows from $\lambda_l/2 \geq 1$ and $\lambda_u/2 \leq 1$.
    By \eqref{eq:lip:descent}, we have $a_{k+1} \leq a_k - b_k + c_k$.
    Applying \cref{quasivariant} now implies the convergence of the sequence $\{ d^2(x^*,x^k) \}_{k \in \N}$, boundedness of $\{ x^k \}_{k \in \N}$, as well as
	\[
		\sum^{\infty}_{k=0}2\tau_k (H(x^{k+1}) - H(x^*)) < \infty.
	\]
	This, with the assumption $\sum_{k=0}^{\infty} \tau_k = \infty$, implies that there exists a subsequence $\{ x^{k_i} \}_{i \in \N}$ of $\{ x^k \}_{k \in \N}$ such that $\lim_{i \to \infty} H(x^{k_i}) = H(x^*)$. Since $\{ x^{k_i} \}_{i \in \N}$ is bounded, it has by local compactness a subsequence converging to a point $\hat{x} \in \widebar{B}(x^*,R)$. By the lower semicontinuity of $H$ it must be a minimiser thus implying that the sequence $\{ d^2(\hat{x},x^k) \}_{k \in \N}$ converges to $0$, since a subsequence of $\{ x^k \}_{k \in \N}$ converges to $\hat{x}$. Hence we have that $\lim_{k \to \infty} x^k = \hat{x}$.
\end{proof}

With these assumptions, this method enjoys the following iteration-complexity bound.

\begin{corollary}[Iteration-complexity, Lipschitz continuous $G$ and $F$]
	\label{itcompllip}
    Under the assumptions of \cref{lip}, for all $N \in \N$, we have
	\[
		H(x^{N}) - H(x^*) \leq \frac{d^2(x^*,x^0) + (\left(\frac{\lambda_l}{2}-1\right)L_{\exp}^2L_F^2 + \left(1 - \frac{\lambda_u}{2}\right)4L^2_G) \sum_{k=0}^{N-1} \tau_k^2}{2 \sum_{k=0}^{N-1} \tau_k}.
	\]

\end{corollary}

\begin{proof}
    Let $y=x^*$  and sum over \cref{eq:lip:descent} from $0$ to $N-1$ to get
	\[
    	\sum_{k=0}^{N-1} 2\tau_k (H(x^{k+1}) - H(x^*)) \leq d^2(x^*, x^0) + \bigg(\left(\frac{\lambda_l}{2}-1\right)L_{\exp}^2L_F^2 + \left(1 - \frac{\lambda_u}{2}\right)4L^2_G \bigg) \sum_{k=0}^{N-1} \tau_k^2.
	\]
    Since $H(x^k)$ is monotone nonincreasing by \cref{prop:proxgrad:monotonicity}, $H(x^{N-1}) \leq H(x^k)$ for all $k=0,1\ldots,N-1$ gives us the result
	\[
        \begin{split}
            2(H(x^N) - H(x^*)) \sum_{k=0}^{N-1} \tau_k &\leq \sum_{k=0}^{N-1} 2\tau_k (H(x^{k+1}) - H(x^*)) \\
            &\leq d^2(x^*, x^0) + \bigg(\left(\frac{\lambda_l}{2}-1\right)L_{\exp}^2L_F^2 + \left(1 - \frac{\lambda_u}{2}\right)4L^2_G \bigg) \sum_{k=0}^{N-1} \tau_k^2.
            \qedhere
        \end{split}
	\]
\end{proof}

\begin{remark}
As noted before, the Lipschitz condition on $G$ is not necessary when the curvature is bounded from above by $0$ (since $1-\lambda_u/2 = 0$).
\end{remark}

With an additional restriction on the diameter of the neighborhood of the minimising point and some penalty on the step sizes (due to said restriction), we can also prove convergence to a minimiser without the rather restrictive Lipschitz assumption on $G$.

\begin{theorem}[Convergence to a minimiser, Lipschitz continuous F]
	\label{nolip}
    Along with \cref{assum:space,assum:func,assum:min}, let $F$ be Lipschitz continuous with factor $L_F$. Let $x^0 \in \widebar{B}(x^*,r)$ for $r>0$ and $R>2r$ such that $R < \pi/(4\sqrt{\overline{\kappa}})$ and $\lambda_u/2 = 2R\sqrt{\overline{\kappa}}/\tan(2R\sqrt{\overline{\kappa}}) > 1/2$ hold (see \cref{assumlambda}). Let $\{ \tau_k \}_{k \in \N}$ be a positive sequence of step sizes satisfying $\sum_{k=0}^{\infty} \tau_k = \infty$, $\sum_{k=0}^{\infty} \tau_k^2 < \infty$, and
	\begin{equation*}
		\tau_k < \min \bigg\{\frac{\lambda_u - 1}{L}, \frac{R - 2r}{2L_F(1 + L_{\exp})} \bigg\}
	\end{equation*}
	for all $k \geq 0$. If $\inf_{x \in X} H(x)$ is attained at some point, then the sequence $\{ x^k \}_{k \in \N}$ created by the iterative process \cref{proxgraditer} converges to some global minimiser of $H$ in $\widebar{B}(x^*,R)$ as $k \to \infty$.
\end{theorem}

\begin{proof}
    The overall outline of the proof is analogous to \cref{lip}, however the constants $c_k$, estimating the final terms of the claim of \cref{keythm}, will be different. In particular, we will estimate $d^2(w^k,x^{k+1})$ differently, to avoid imposing the Lipschitz requirement on $G$.

    Since the inequalities $\lambda_u/2 > 1/2$ and $L\tau_k < \lambda_u - 1$ are strict, there exists an $\epsilon > 0$ such that $\lambda_u/2 > (\epsilon + 1)/(\epsilon + 2)$ and $L\tau_k < \lambda_u(\epsilon + 2)/2 - 1 - \epsilon$. By the triangle and Young's inequalities, we have
    \[
    \begin{split}
        d^2(w^k,x^{k+1}) &\leq d^2(w^k,x^k) + d^2(x^k,x^{k+1}) + 2d(x^k,x^{k+1})d(w^k,x^k) \\
        &\leq (1 + \epsilon^{-1})d^2(w^k,x^k) + (1 + \epsilon)d^2(x^k,x^{k+1}).
    \end{split}
    \]
    This and \cref{keythm} with $y=x^*$ and $L\tau_k < \lambda_u(\epsilon + 2)/2 - 1 - \epsilon$ give, for $x^k, w^k, x^{k+1} \in \widebar{B}(x^*,R)$,
	\[
	\begin{split}
		d^2(x^*,x^{k+1}) &\leq d^2(x^*,x^k) + 2\tau_k (H(x^*) - H(x^{k+1})) + (L\tau_k - \frac{\lambda_u}{2})d^2(x^k,x^{k+1}) + \left(\frac{\lambda_l}{2}-1\right)d^2(x^k,w^k) \\
		& \quad + \left(1 - \frac{\lambda_u}{2}\right)d^2(w^k,x^{k+1}) \\
		&\leq d^2(x^*,x^k) + 2\tau_k (H(x^*) - H(x^{k+1})) + (L\tau_k - \frac{\lambda_u}{2})d^2(x^k,x^{k+1}) + \left(\frac{\lambda_l}{2}-1\right)d^2(x^k,w^k) \\
		& \quad + (1-\frac{\lambda_u}{2})(1 + \epsilon^{-1})d^2(w^k,x^k) + (1-\frac{\lambda_u}{2})(1 + \epsilon)d^2(x^k,x^{k+1}) \\
		&\leq d^2(x^*,x^k) - 2\tau_k (H(x^{k+1}) - H(x^*)) + (\frac{\lambda_l}{2} + \epsilon^{-1} - \frac{\lambda_u(1 + \epsilon)}{2\epsilon})d^2(x^k,w^k).
	\end{split}
	\]
    Using the Lipschitz continuity of the exponential map and $F$ grants us the estimate $d^2(x^k,w^k) \leq L_{\exp}^2 \tau_k^2 L_F^2$, further giving
	\begin{equation*}
		d^2(x^*,x^{k+1}) \leq d^2(x^*,x^k) - 2\tau_k (H(x^{k+1}) - H(x^*)) + c_k,
	\end{equation*}
    i.e., \eqref{eq:lip:descent}, now with
	\[
		c_k = \bigg( \left(\frac{\lambda_l}{2} + \epsilon^{-1} - \frac{\lambda_u(1 + \epsilon)}{2\epsilon}\right)L_{\exp}^2L_F^2 \bigg)\tau_k^2.
	\]
	Note that $c_k \geq 0$ for all $k \geq 0$, since $ \lambda_l/2 \geq 1$, $\lambda_u/2 \leq 1$, and thus $\lambda_l/2 + \epsilon^{-1} - \lambda_u(1 + \epsilon)/(2\epsilon) \geq 1 + \epsilon^{-1} - (1 + \epsilon)/\epsilon = 0$.

    The rest of the proof is exactly as that of \cref{lip}.
\end{proof}

\begin{remark}
	\label{assumlambda}
    The assumption $\lambda_u/2 > 1/2$ is necessitated by the requirements  $L\tau_k + 1 + - \lambda_u < 0$ and $L\tau_k > 0$. This can be met, as the factor $\lambda_u/2 = 2R\sqrt{\overline{\kappa}}/\tan(2R\sqrt{\overline{\kappa}})$ can be made to be as close to $1$ as we wish by making $R$ smaller (and hence $\diam \widebar{B}(x^*,R)$ smaller). Preferably, the factor would be as small as possible to have the diameter be as large as possible.
\end{remark}

For these step size restrictions, we have the following iteration-complexity bound for the forward-backward method.

\begin{corollary}[Iteration-complexity, Lipschitz continuous $F$]
	\label{iter}
	Under the assumptions of \cref{nolip} and some choice of $\epsilon > 0$, for all $N \in \N$, we have
	\[
		H(x^{N}) - H(x^*) \leq L_F \frac{d^2(x^0,x^*) + (\frac{\lambda_l}{2} + \epsilon^{-1} - \frac{\lambda_u(1 + \epsilon)}{2\epsilon})\sum_{k=0}^{N-1}\alpha_k^2}{2 \sum_{k=0}^{N-1} \alpha_k},
	\]
    where $\alpha_k = \tau_k |\nabla_-F|(x^k) L_{\exp} $, for $k=0,1,\ldots$.
\end{corollary}

\begin{proof}
    Same as \cref{iter1} with added monotonicity of $H(x^k)$ as in \cref{itcompllip}.
\end{proof}

\begin{remark}
	As noted before, in the presence of a nonzero curvature bound, the square summability of the step sizes seems to always be necessary in obtaining convergence results; square summability ensures that the step sizes are small enough while the nonsummability of the general sequence keeps them from being too small (guaranteeing convergence).
\end{remark}

For a more general case of the key inequality, we may consider the instance when $G$ and $F$ can also be strongly convex.

\begin{lemma}[Key lemma, strongly convex $G$ and $F$]
	\label{keythmlambda}
	Along with \cref{assum:space,assum:func}, let $G$ and $F$ be $\lambda$-convex with factors $\lambda_G, \lambda_F \geq 0$, respectively. Fix the step size $\tau > 0$ and a closed, geodesically convex set $A \subset X$ ($\diam A < \pi/(2\sqrt {\overline{\kappa}})$, if $\overline{\kappa} > 0$) containing a nonempty sublevel set of $G + F$. The key inequality \cref{key} now reads
\begin{equation}
	\label{lc1}
	\begin{split}
		(1 + \lambda_G \tau)d^2(y,J(x)) &\leq (1 - \lambda_F \tau)d^2(y,x) + 2\tau (H(y) - H(J(x))) + (L\tau - \frac{\lambda_u}{2})d^2(x,J(x)) \\
		&+ \left(\frac{\lambda_l}{2}-1\right)d^2(x,w) + \left(1 - \frac{\lambda_u}{2}\right)d^2(w, J(x))
	\end{split}
\end{equation}
for all $x, y \in A$ satisfying $w, J(x) \in A$, where $H(x) = G(x) + F(x)$, $w = \exp_{x^k}(\tau \nabla(-F)(x))$, $\lambda_u/2 = \sqrt{\overline{\kappa}}\diam A/\tan(\sqrt{\overline{\kappa}}\diam A)$, and $\lambda_l/2 = \sqrt{-\underline{\kappa}}\diam A/\tanh(\sqrt{-\underline{\kappa}}\diam A)$.

Furthermore, when $x^*$ is some global minimiser of $H$, we have
\begin{equation}
	\label{lc2}
	\begin{split}
		(1 + \lambda_G \tau + \lambda_H \tau)d^2(x^*,J(x)) &\leq (1 - \lambda_F \tau)d^2(x^*,x) + (L\tau - \frac{\lambda_u}{2})d^2(x,J(x)) \\
		&+ \left(\frac{\lambda_l}{2}-1\right)d^2(x,w) + \left(1 - \frac{\lambda_u}{2}\right)d^2(w, J(x)),
	\end{split}
\end{equation}
where $\lambda_H$ is the convexity factor of $H$ ($\lambda_H = \lambda_G + \lambda_F$).
\end{lemma}

\begin{proof}
    For a $\lambda_G$-convex $G$, we have by \cref{subdiff}
    \[
        D_{\prox_{\tau G}(x)}G(\gamma^+(0)) \leq G(y) - G(\prox_{\tau G}(x)) - (\lambda_G/2) d^2(y, \prox_{\tau G}(x)),
    \]
    so \cref{key1new} now improves to
	\[
		d^2(y, \prox_{\tau G}(x)) \leq d^2(y,x) + 2\tau (G(y) - G(\prox_{\tau G}(x))) - \frac{\lambda_u}{2}d^2(x, \prox_{\tau G}(x)) - \lambda_G \tau d^2(y, \prox_{\tau G}(x)).
	\]
    Analogously \cref{key2new} improves for $\lambda_F$-convex $F$ to
	\[
		d^2(y, \desc_{\tau F}(x)) \leq d^2(y,x) + 2\tau (F(y) - F(x)) + \frac{\lambda_l}{2} d^2(x, \desc_{\tau F}(x)) - \lambda_F \tau d^2(y, x)
	\]
    By following the proof of \cref{keythm} with these estimates now gives us \cref{lc1}. For \cref{lc2}, let $\lambda_H$ be the convexity factor of $H$ and $x^*$ its minimiser. For a geodesic $\gamma$ between $J(x)$ and $x^*$, we have $H(\gamma(t)) \leq (1-t)H(J(x)) + tH(x^*) - (\lambda_H /2)t(1-t)d^2(J(x),x^*)$ by $\lambda$-convexity and $H(x^*) \leq H(\gamma(t))$ by minimality, for all $t \in [0,1]$ and any $J(x) \in A$. Dividing by $1-t$ and letting $t \to 1$ gives us
	\[
    	\frac{\lambda_H}{2}d^2(J(x), x^*) \leq H(J(x)) - H(x^*).
	\]
    The claim now follows by choosing $y = x^*$ in \cref{lc1}.
\end{proof}

As strongly convex functions in complete geodesic spaces have unique minima \cite[Lemma 14.4]{petruninfoundations}, one can again show using \cref{keythmlambda}, that the forward-backward method converges to the unique minimiser of $H$ (one might even allow $\lambda_F < 0$ as long as $\lambda_H > 0$). As an example, we again take $G$ and $F$ to be Lipschitz continuous with $G$ being convex and $F$ strongly convex.

\begin{theorem}[Convergence to a minimiser, Lipschitz continuous $G$ and $F$, and strongly convex $F$]
	Along with the assumptions of \cref{lip}, let $F$ be strongly convex with factor $\lambda_F > 0$ and $\tau_k \lambda_F < 1$ for all $k \geq 0$. If $\inf_{x \in X} H(x)$ is attained at some point, then the sequence $\{ x^k \}_{k \in \N}$ created by the iterative process \cref{proxgraditer} converges to the unique minimiser of $H$ in $\widebar{B}(x^*,R)$ as $k \to \infty$.

	Furthermore, we have
	\begin{equation}
	\label{linear}
		d^2(x^{k+1}, x^*) \leq \prod_{i=0}^{k}\frac{1-\tau_i \lambda_F}{1+\tau_i \lambda_F} d^2(x^0,x^*) + \sum_{i=0}^{k} \bigg( \bigg( (\frac{\lambda_l}{2} - 1 )\frac{L_{\exp}^2L_F^2\tau_i^2}{1+\tau_i \lambda_F} + \left(1 - \frac{\lambda_u}{2}\right) \frac{4L^2_G\tau_i^2}{1+\tau_i \lambda_F}\bigg) \prod_{j=i+1}^{k}\frac{1-\tau_j \lambda_F}{1+\tau_j \lambda_F} \bigg),
	\end{equation}
	with the convention that $\prod_{j=m}^{n}$ is the empty product when $m>n$, i.e., equal to $1$.
\end{theorem}

The proof follows the same locality and convergence arguments as in the proof of \cref{lip}, while uniqueness of the minimiser is due to the strong convexity of $H$. The expression \cref{linear} follows by iterating \cref{lc2} with the Lipschitz estimates from \cref{lip}.

\subsection{Inexact method}
\label{sec:inexact}

To close out the this section, we briefly outline how the method can be made inexact in both gradient and proximal steps. After the relevant changes, the proofs go through mutatis mutandis.

Instead of the exact optimality condition of \cref{lemma:proxopt} for the proximal step, we assume an $\epsilon$-approximate minimum so that the lower bound becomes $-\epsilon_G \leq 0$ rather than $0$ in \cref{eq:proxopt}. This trickles down through \cref{eq:key1new:item1} to \cref{eq:key1new,key} as $2\tau \epsilon_G$ on their right hand-sides.

To introduce inexactness to the gradient step, we replace $\nabla (-F)(x)$ with $g_x$ in \cref{gradexpresolv} and assume that $d_{\angle_x}(g_x, \nabla (-F)(x)) \leq \epsilon_F$, where $\epsilon_F \geq 0$. Now the application of
	\[
		\begin{split}
		-2\tau \langle g_x, \gamma^+(0) \rangle &= -2\tau \langle \nabla (-F)(x), \gamma^+(0) \rangle + 2\tau \langle \nabla (-F)(x) - g_x, \gamma^+(0) \rangle\\
		&\leq -2\tau \langle \nabla (-F)(x), \gamma^+(0) \rangle + 2\tau d_{\angle_x}(g_x, \nabla (-F)(x)) d(x, y)\\
		&\leq -2\tau \langle \nabla (-F)(x), \gamma^+(0) \rangle + 2\tau \epsilon_F d(x,y) \leq -2\tau \langle \nabla (-F)(x), \gamma^+(0) \rangle + 2\tau \epsilon_F \diam A
		\end{split}
	\]
in \cref{eq:key2new:item1} gives us \cref{eq:key2new,key} with $2\tau \epsilon_F \diam A$ on their right hand-sides.

Similarly, the monotonicity result \cref{prop:proxgrad:monotonicity} gets replaced with quasi-monotonicity as it has the additional term $\epsilon_{G_k} + \epsilon_{F_k} \diam A$ on the right side of \cref{eq:proxgrad:monotonicity}. For convergence of iterates and decrease of function values, we can assume that $\epsilon_{G_k} + \epsilon_{F_k} \diam A \leq \alpha d^2(x^k, x^{k+1})$, for some constant $\alpha \geq 0$ such that $\lambda_u - 2L\tau_k - 4\alpha \tau_k > 0$. The a priori locality \cref{localstep} is updated with the new step size bound $\tau_k \leq (R^2 - 2Rr)(2\epsilon_{G_k} + 2R(L_F + L_F L_{\exp} + L_{\exp}\epsilon_{F_k}))^{-1}$, where the necessary changes are made in \cref{eq:localstep:prox,localwbound} accordingly.

Finally, the main convergence results, \cref{lip,nolip}, follow from the aforementioned modified estimates by further assuming that the errors $\epsilon_{G_k}, \epsilon_{F_k}$ are summable, since then $\sum_{k} \tau_k \epsilon_{G_k}, \sum_{k} \tau_k \epsilon_{F_k} < \infty$ by H{\"o}lder. Note that all these inexact results reduce to the ones we have presented above, when the inexactness disappears.

\section{Examples}
\label{sec:examples}

We test our algorithm on a Lasso-like problem on two surfaces: the surface of a cube and that of a right circular cylinder capped from both ends with a flat disc of the same radius. Both of these surfaces have curvature bounded from below by $0$. However, the cube fails to be upper bounded at the vertices and the capped cylinder at the two bands where the discs meet the cylinder. For our algorithm, the cube does not present a problem as the vertices can be removed (geodesics do not go through them) or avoided by numerical perturbations; specifically, we can use inexact steps of \cref{sec:inexact} for the latter. This does not work for the capped cylinder as all of the interesting behaviour happens when an edge is crossed. Instead of this surface, we could have experimented with capsule (see paragraph after \cref{bilat}) since it satisfies all of the necessary assumptions. However, we do not find that example very interesting, so we decided to see how the method performs on the flat capped cylinder, despite the latter not satisfying all the assumptions.

For a regularisation parameter $\lambda>0$, a choice of “origin” $z \in X$, and data points $\{y_i\}_{i=1}^n \subset X$, the objective function on both surfaces will be
\begin{equation*}
	H(x) = F(x) + G(x) = \sum_{i=1}^n \frac{1}{2} d^2(x, y_i) + \lambda d(x, z),
\end{equation*}
To apply the forward-backward splitting method to minimise this function, we need to calculate the gradient of $F$ and the proximal operator of $G$.
The latter will be a type of soft thresholding operator.
Practically, to compute these maps, we need to form the exponential and logarithmic maps of the surfaces.
The software implementation of our experiments (in Rust) is available on Zenodo \cite{valkonenvonkoch2024nonriemannian-codes}.

\subsection{Auxiliary lemma}

For $y,z \in B(x,D_{\overline{\kappa}})$, where $D_{\overline{\kappa}}$ is the diameter of the model space (see \cref{sec:alexandrov}), we have by \cite[Chapter 13, sections D and E]{petruninfoundations}, the gradients
\begin{equation}
	\label{distancegrad}
	\nabla F(x) = -\sum_{i=1}^n \log_x y_i \quad \text{and} \quad \nabla G(x) = -\lambda \frac{\log_x z}{d(x,z)}, \enskip \text{if} \; x \neq z.
\end{equation}
The gradient step can now be written as
    \[
        \desc_{\tau F}(x) = \exp_x(\tau \nabla (-F)(x)) = \exp_x(\tau \sum_{i=1}^n \log_x y_i),
    \]
whereas the proximal step needs a little bit more work and has the following closed-form solution.

\begin{lemma}
    Let $G(x) = \lambda d(x,z)$, with $x,z \in A \subset X$, $\diam A < \pi/(2{\sqrt{\overline{\kappa}}})$, and $\lambda > 0$. Then
    \begin{equation*}
        \prox_{\tau G}(x) = y^* =
            \begin{cases}
                z, \quad & \text{if} \; d(x,z) \leq \lambda \tau;\\
                \exp_x(\frac{\lambda \tau}{d(x,z)}\log_x z), \quad & \text{if} \; d(x,z) > \lambda \tau.
            \end{cases}
    \end{equation*}
\end{lemma}
For consistency, we could also write $z = \exp_x(\log_x z)$.

\begin{proof}
    Consider first $d(x,z) \leq \lambda \tau$. Let $y \in X$. We have by the triangle inequality
    \[
        (d(x,z) - d(y,z))^2 \leq d^2(x,y) \leq d^2(x,y) + d^2(y,z)
    \]
    so that
    \[
        d^2(x,z) \leq 2d(x,z)d(y,z) + d^2(x,y).
    \]
	Using the assumption $d(x,z) \leq \lambda \tau$, and adding $2\lambda \tau d(z,z)=0$ on both sides leads to
    \[
        2\lambda \tau d(z,z) + d^2(x,z) \leq 2\lambda \tau d(y,z) + d^2(x,y)
        \quad\text{for all}\quad
        y \in X.
    \]
    This is to say, $\text{prox}_{\tau G}(x) = z$.

    Now, suppose $d(x,z) > \lambda \tau$. By \cref{distancegrad}, for any $y \ne z$, we have
    \[
        \nabla (G(y) + (2\tau)^{-1}d^2(x,y))(y) = -\lambda (d(y,z))^{-1} \log_y z - \tau^{-1} \log_y x,
    \]
    which for $y=y^* \ne z$ to be a minimiser, must equal to $0$, that is,
    \begin{equation}\label{softthresh}
        \log_{y^*} x = -\lambda \tau (d(y^*,z))^{-1} \log_{y^*} z.
    \end{equation}
    By the assumption $x,z \in A$, we can write the geodesic $\gamma(t) = x \#_t z$, for $t \in [0,1]$, as $\gamma(t) = \exp_x(t\log_x z)$. Since the lengths of tangent vectors are always positive and the geodesic $\gamma$ is unique, the equivalence of the left- and right-hand sides of \cref{softthresh} implies that $y^*$ lies on the geodesic $\gamma$, i.e. $y^* = \exp_x(t^*\log_x z)$, for some $0<t^*<1$.
    More explicitly, we have by \cref{softthresh}
    \[
        d(y^*, x) = d(\exp_x(t^*\log_x z), x) = t^*|\log_x z| \implies t^*|\log_x z| = \lambda \tau \implies t^* = \lambda \tau / |\log_x z|,
    \]
    which, by assumption $d(x,z) > \lambda \tau > 0$, satisfies $0<t^*<1$.
    Although the formula $y^* = \exp_x(t^*\log_x z)$ would still hold with $t^*=0$ in the case $y^*=z$, we have just shown that this case never happens.
\end{proof}

\subsection{Cube}
\label{sec:examples:cube}

\begin{figure}[t]
	\centering
	\tdplotsetmaincoords{50}{25}
	\begin{tikzpicture}
			[tdplot_main_coords,
				cube/.style={very thick,black},
				grid/.style={very thin,gray},
				axis/.style={->,blue,thick}]

		\foreach \x in {-1.5,-1,...,1.5}
			\foreach \y in {-1.5,-1,...,1.5}
			{
				\draw[grid] (\x,-1.5) -- (\x,1.5);
				\draw[grid] (-1.5,\y) -- (1.5,\y);
			}

		\draw[axis] (-1,-1,0) -- (2,-1,0) node[anchor=west]{$x$};
		\draw[axis] (-1,-1,0) -- (-1,2,0) node[anchor=west]{$y$};
		\draw[axis] (-1,-1,0) -- (-1,-1,3) node[anchor=west]{$z$};

		\draw[cube] (-1,-1,0) -- (-1,1,0) -- (1,1,0) -- (1,-1,0) -- cycle;
		\draw[cube] (-1,-1,2) -- (-1,1,2) -- (1,1,2) -- (1,-1,2) -- cycle;

		\draw[cube] (-1,-1,0) -- (-1,-1,2);
		\draw[cube] (-1,1,0) -- (-1,1,2);
		\draw[cube] (1,1,0) -- (1,1,2);
		\draw[cube] (1,-1,0) -- (1,-1,2);

		\filldraw[black] (-1,-1) circle (2pt) node[anchor=east]{$O$};
		\filldraw[black] (1,1) circle (2pt) node[anchor=west]{$(1,1,0)$};
		\filldraw[black] (1,1,2) circle (2pt) node[anchor=west]{$(1,1,1)$};
		\filldraw[black] (-1,1,2) circle (2pt) node[anchor=south]{$(0,1,1)$};
		\filldraw[black] (-1,-1,2) circle (2pt) node[anchor=east]{$(0,0,1)$};
	\end{tikzpicture} \hspace{1em}
	\tdplotsetmaincoords{60}{20}
	\begin{tikzpicture}
			[tdplot_main_coords,
				cube/.style={very thick,black},
				axis/.style={->,blue,thick}]

		\draw[axis] (-1,-1,0) -- (-1,-1,1) node[anchor=east]{$y_4$};
		\draw[axis] (-1,-1,0) -- (0.1,-1,0) node[anchor=north]{$x_4$};
		\draw[axis] (-1,1,0) -- (-1,0.1,0) node[anchor=west]{$x_3$};
		\draw[axis] (-1,1,0) -- (-1,1,1) node[anchor=west]{$y_3$};
			\draw[axis] (1,1,0) -- (1,1,1) node[anchor=west]{$y_5$};
		\draw[axis] (1,1,0) -- (0,1,0) node[anchor=south]{$x_5$};
			\draw[axis] (1,-1,0) -- (1,-1,1) node[anchor=west]{$y_2$};
		\draw[axis] (1,-1,0) -- (1,0,0) node[anchor=west]{$x_2$};

		\draw[cube] (-1,-1,0) -- (-1,1,0) -- (1,1,0) -- (1,-1,0) -- cycle;
		\draw[cube] (-1,-1,2) -- (-1,1,2) -- (1,1,2) -- (1,-1,2) -- cycle;

		\draw[cube] (-1,-1,0) -- (-1,-1,2);
		\draw[cube] (-1,1,0) -- (-1,1,2);
		\draw[cube] (1,1,0) -- (1,1,2);
		\draw[cube] (1,-1,0) -- (1,-1,2);

		\draw(1,0,1) node{$2$};
		\draw(0,-1,1) node{$4$};
		\draw(0,0,2) node{$6$};
		\draw[gray!100](0,0,0) node{$1$};
		\draw[gray!100](-1,0,1) node{$3$};
		\draw[gray!100](0,1,1) node{$5$};
	\end{tikzpicture} \hspace{2em}
	\tdplotsetmaincoords{60}{20}
	\begin{tikzpicture}
			[tdplot_main_coords,
				cube/.style={very thick,black},
				axis/.style={->,blue,thick}]

		\draw[axis] (-1,-1,0) -- (0.1,-1,0) node[anchor=north]{$x_1$};
		\draw[axis] (-1,-1,0) -- (-1,0.1,0) node[anchor=west]{$y_1$};
		\draw[axis] (-1,-1,2) -- (0.1,-1,2) node[anchor=north]{$x_6$};
		\draw[axis] (-1,-1,2) -- (-1,0.1,2) node[anchor=east]{$y_6$};

		\draw[cube] (-1,-1,0) -- (-1,1,0) -- (1,1,0) -- (1,-1,0) -- cycle;
		\draw[cube] (-1,-1,2) -- (-1,1,2) -- (1,1,2) -- (1,-1,2) -- cycle;

		\draw[cube] (-1,-1,0) -- (-1,-1,2);
		\draw[cube] (-1,1,0) -- (-1,1,2);
		\draw[cube] (1,1,0) -- (1,1,2);
		\draw[cube] (1,-1,0) -- (1,-1,2);

		\draw(1,0,1) node{$2$};
		\draw(0,-1,1) node{$4$};
		\draw(0,0,2) node{$6$};
		\draw[gray!100](0,0,0) node{$1$};
		\draw[gray!100](-1,0,1) node{$3$};
		\draw[gray!100](0,1,1) node{$5$};
	\end{tikzpicture}
	\caption{A unit cube with labeled faces and their local coordinate systems.}
	\label{fig:cube}
\end{figure}
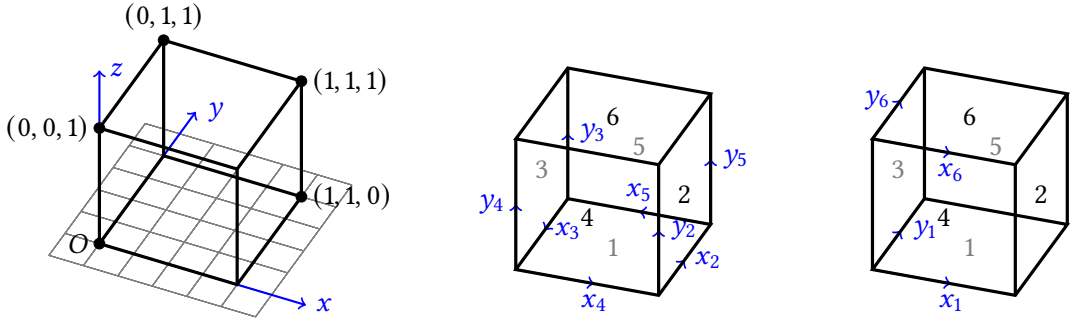

Let our surface be a unit cube with labeled faces and the origin as one of the vertices. In addition, let each face $F$ have their own fixed oriented local two-dimensional coordinate system denoted by $(x_F,y_F)_F \in [0,1]^2$, where $F \in \{1,2,3,4,5,6\}$, as in \cref{fig:cube}. We shall use the triple $(x, y, z)$ for the standard coordinates in three-dimensional Euclidean space.
As the surface of a convex body, the curvature of the cube is bounded from below by $\underline{\kappa}=0$ \cite[Theorem 10.2.6]{burago2001course}.
By the Reshetnyak gluing theorem \cite[Theorem 9.1.21]{burago2001course}, \emph{away from the corners}, but including the remainder of the edges, the curvature is  also bounded from above by $\overline{\kappa}=0$.
Since geodesics are not unique at the corners, the curvature cannot be bounded from above near the corners \cite[Theorem 9.1.17]{burago2001course}.
Practically, we will never reach the corners\footnote{\cref{lemma:proxopt}\cref{key1new}  \cref{keythm} }, so our convergence theory is applicable.
Alternatively, we could smoothen the corners (but, importantly, not the edges) to have  the entire manifold bounded from both above and below.

Geodesics on the 3D cube form straight lines that do not go through vertices \cite[Lemma 4.1]{shahironshortest}, when the cube is unfolded into a two-dimensional surface. To determine the geodesics, we need transformations between local coordinate systems, which can then be composed together to make longer chains between faces. A coordinate transformation from one face to another is made by rotating (with respect to the common edge) the target face onto the same plane as the starting face (keeping this face fixed) such that there is no overlap. \Cref{table:cube} contains all coordinate transformations between adjacent faces as well as their standard three-dimensional coordinates in each face. Distances between points can be calculated with the help of \cref{table:cube} by transforming the coordinate of the endpoint to that of the starting point's system and using the Pythagorean theorem for all of the possible paths between these points. The possible transformations also facilitates the formation of the exponential and logarithm maps. An illustration of geodesics between two points on faces $1$ and $2$ in different unfoldings can be seen in \cref{fig:cubeunfold}. We illustrate the numerical behaviour of our forward-backward algorithm on the cube in \cref{fig:numerical:cube}.

\begin{table}[t]
    \caption{Coordinate transformations between the faces, with $x,y \in [0,1]$, as well as their   standard three-dimensional coordinates. Identity transformations are denoted in grey. Note that the empty places correspond to non-adjacent faces.}
	\label{table:cube}
	\centering
	\[
		\begin{blockarray}{ccccccc}
        	\R^3 & (x,y,0) & (1,x,y) & (0,1-x,y) & (x,0,y) & (1-x,1,y) & (x,y,1) \\
        	\cmidrule{1-7}
        	\text{Face \#} & 1 & 2 & 3 & 4 & 5 & 6 \\
		\begin{block}{c(cccccc)}
        	1 & \makecell{\cellcolor{lightgray}(x,y)_1} & (y+1,x)_1 & (-y,1-x)_1 & (x,-y)_1 & (1-x,y+1)_1 & \\
        	2 & (y,x-1)_2 & \makecell{\cellcolor{lightgray}(x,y)_2} &  & (x-1,y)_2  & (x+1,y)_2 & (y,2-x)_2 \\
        	3 & (1-y,-x)_3 &  & \makecell{\cellcolor{lightgray}(x,y)_3} & (x+1,y)_3 & (x-1,y)_3 & (1-y,x+1)_3 \\
        	4 & (x,-y)_4 & (x+1,y)_4 & (x-1,y)_4 & \makecell{\cellcolor{lightgray}(x,y)_4} & & (x,y+1)_4\\
        	5 & (1-x,y-1)_5 & (x-1,y)_5 & (x+1,y)_5 & & \makecell{\cellcolor{lightgray}(x,y)_5} & (1-x,2-y)_5 \\
        	6 & & (2-y,x)_6 & (y-1,1-x)_6 & (x,y-1)_6 & (1-x,2-y)_6 & \makecell{\cellcolor{lightgray}(x,y)_6} \\
		\end{block}
		\end{blockarray}
	\]
\end{table}

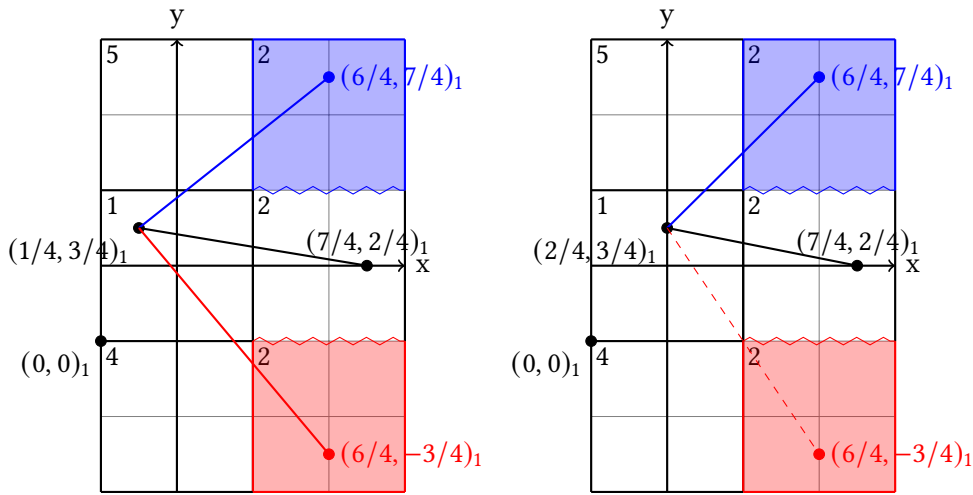
\begin{figure}
	\centering
	\begin{tikzpicture}
    	\draw[step=1cm,gray,very thin] (-1,-3) grid (3,3);

		\draw[thick,-] (-1,-1) -- (1,-1);
		\draw[thick,-] (-1,1) -- (1,1);
		\draw[thick,-] (1,3) -- (1,-3);
		\draw[thick,-] (-1,-3) -- (-1,3);
		\draw[thick,-,red] (3,-3) -- (3,-1);
		\draw[thick,-] (3,-1) -- (3,1);
		\draw[thick,-,blue] (3,1) -- (3,3);
		\draw[thick,-,red] (1,-3) -- (1,-1);
		\draw[thick,-] (1,-1) -- (1,1);
		\draw[thick,-,blue] (1,1) -- (1,3);
		\draw[thick,-] (-1,3) -- (1,3);
		\draw[thick,-,blue] (1,3) -- (3,3);
		\draw[thick,-] (-1,-3) -- (1,-3);
		\draw[thick,-,red] (1,-3) -- (3,-3);

		\draw(-0.85,0.8) node{$1$};
		\draw(-0.85,2.8) node{$5$};
		\draw(-0.85,-1.2) node{$4$};
		\draw(1.15,0.8) node{$2$};
		\draw(1.15,2.8) node{$2$};
		\draw(1.15,-1.2) node{$2$};

		\draw[decorate,decoration={zigzag, amplitude = 0.5mm},blue] (1,1) -- (3,1);
		\draw[decorate,decoration={zigzag, amplitude = 0.5mm},red] (1,-1) -- (3,-1);

		\draw[thick,->] (-1,0) -- (3,0) node[anchor=west] {x};
		\draw[thick,->] (0,-3) -- (0,3) node[anchor=south] {y};

		\filldraw[black] (-1,-1) circle (2pt) node[anchor=north east]{$(0,0)_1$};
		\filldraw[black] (-0.5,0.5) circle (2pt) node[anchor=north east]{$(1/4,3/4)_1$};
		\filldraw[blue] (2,2.5) circle (2pt) node[anchor=west]{$(6/4,7/4)_1$};
		\filldraw[black] (2.5,0) circle (2pt) node[anchor=south]{$(7/4,2/4)_1$};
		\filldraw[red] (2,-2.5) circle (2pt) node[anchor=west]{$(6/4,-3/4)_1$};

		\draw[thick,-,blue] (-0.5,0.5) -- (2,2.5);
		\draw[thick,-] (-0.5,0.5) -- (2.5,0);
		\draw[thick,-,red] (-0.5,0.5) -- (2,-2.5);

		\node[fill=blue, fill opacity=0.3, scale=7.8] at (2,2) {};
		\node[fill=red, fill opacity=0.3, scale=7.8] at (2,-2) {};
	\end{tikzpicture}
	\begin{tikzpicture}
    	\draw[step=1cm,gray,very thin] (-1,-3) grid (3,3);

		\draw[thick,-] (-1,-1) -- (1,-1);
		\draw[thick,-] (-1,1) -- (1,1);
		\draw[thick,-] (1,3) -- (1,-3);
		\draw[thick,-] (-1,-3) -- (-1,3);
		\draw[thick,-,red] (3,-3) -- (3,-1);
		\draw[thick,-] (3,-1) -- (3,1);
		\draw[thick,-,blue] (3,1) -- (3,3);
		\draw[thick,-,red] (1,-3) -- (1,-1);
		\draw[thick,-] (1,-1) -- (1,1);
		\draw[thick,-,blue] (1,1) -- (1,3);
		\draw[thick,-] (-1,3) -- (1,3);
		\draw[thick,-,blue] (1,3) -- (3,3);
		\draw[thick,-] (-1,-3) -- (1,-3);
		\draw[thick,-,red] (1,-3) -- (3,-3);

		\draw(-0.85,0.8) node{$1$};
		\draw(-0.85,2.8) node{$5$};
		\draw(-0.85,-1.2) node{$4$};
		\draw(1.15,0.8) node{$2$};
		\draw(1.15,2.8) node{$2$};
		\draw(1.15,-1.2) node{$2$};

		\draw[decorate,decoration={zigzag, amplitude = 0.5mm},blue] (1,1) -- (3,1);
		\draw[decorate,decoration={zigzag, amplitude = 0.5mm},red] (1,-1) -- (3,-1);

		\draw[thick,->] (-1,0) -- (3,0) node[anchor=west] {x};
		\draw[thick,->] (0,-3) -- (0,3) node[anchor=south] {y};

		\filldraw[black] (-1,-1) circle (2pt) node[anchor=north east]{$(0,0)_1$};
		\filldraw[black] (0,0.5) circle (2pt) node[anchor=north east]{$(2/4,3/4)_1$};
		\filldraw[blue] (2,2.5) circle (2pt) node[anchor=west]{$(6/4,7/4)_1$};
		\filldraw[black] (2.5,0) circle (2pt) node[anchor=south]{$(7/4,2/4)_1$};
		\filldraw[red] (2,-2.5) circle (2pt) node[anchor=west]{$(6/4,-3/4)_1$};

		\draw[thick,-,blue] (0,0.5) -- (2,2.5);
		\draw[thick,-] (0,0.5) -- (2.5,0);
		\draw[dashed,-,red] (0,0.5) -- (2,-2.5);

		\node[fill=blue, fill opacity=0.3, scale=7.8] at (2,2) {};
		\node[fill=red, fill opacity=0.3, scale=7.8] at (2,-2) {};
	\end{tikzpicture}
	\caption{Two copies of an unfolding of the cube with the labels of the faces marked on their upper left corners. Included are three colour-coded paths from face $1$ to face $2$, with different starting points but same endpoints, in the coordinate system of face $1$. The paths on the left are all geodesics, whereas the dashed red path on the right is not, as it goes through the vertex $(1,0)_1$. Note that the two zigzag edges between faces $2$ are not real but merely artefacts of the drawing.}
	\label{fig:cubeunfold}
\end{figure}

\subsection{Cylinder}

For our second example, we consider the capped circular cylinder
\[
    C_{r,h} \defeq
    \{
        (t\cos \phi, t\sin \phi, z)
        \mid
        \phi \in [0, 2\pi),\, 0 \le t \le r,\, -h \le z \le h,\,
        (t-r)(h-\abs{z})=0
    \} \subset \R^3
\]
of height $2h$ and radius $r$.
Again, as the surface of a convex body, the curvature is bounded from below by $\underline{\kappa}=0$ \cite[Theorem 10.2.6]{burago2001course}.
\emph{Away from the edges}, i.e., independently on the cap and the side, the curvature is also bounded from above  $\overline{\kappa}=0$. This can be easily seen from the definition.
Since geodesics are not unique near the edge between the cylindrical surface and the cap, the curvature cannot be bounded from above at the edge \cite[Theorem 9.1.17]{burago2001course}.
Thus, our convergence theory only applies if, after a finite number of steps, the geodesics joining the iterates no longer cross the edges of the caps, all the iterates staying on either the side or one of the caps.
The computationally more difficult “total surrogate” approach of \cref{rem:total-surrogate} would be applicable without such a restriction.
Moreover, the same considerations as for the cube in \cref{sec:examples:cube}, apply to a polyhedral approximation of the cylinder.

We will equivalently work with the cylindrical coordinates $(z, \phi, t)$.
Geodesics between two points $p_1=(z_1,\phi_1, r)$ and $p_2=(z_2,\phi_2, r)$ on the side are either helices of the form
\[
	\gamma(\phi) = \bigg (r \cos \phi, r \sin \phi, \frac{z_1-z_2}{\phi_1-\phi_2}\phi + \frac{\phi_1 z_2 - \phi_2 z_1}{\phi_1 - \phi_2} \bigg )
\]
or, in the degenerate cases, straight lines or arcs of circles (this follows from Clairaut's Theorem \cite[Proposition 9.3.2]{pressley2010elementary}), whereas geodesics on the disk-shaped caps are simply straight lines. Geodesics on the whole surface are thus concatenations of these types such that the geodesic forms a straight line in an unfolding of the cylinder \cite[Chapter III, Theorem 3]{kutateladze2005d}; the side forms a rectangle and the disk-shaped caps are tangent to this on opposite sides (see \cref{fig:cyltangent}). The two difficult types of geodesics on the whole surface are: a path going over the whole side or a whole cap. A simple example shows that the second case is indeed possible and that instead of a helix on the cylindrical surface, we sometimes take a shortcut through a cap.
\begin{enumerate}
	\item A geodesic starting from $p_1$ at the top cap and ending at $p_2$ at the bottom cap. Points $x_1$ and $x_2$, where the geodesic crosses the top and bottom caps, respectively, need to be chosen such that the curve going through the points $p_1, x_1, x_2$ and $p_2$ is straight in the unfolded picture.
	\item A geodesic between points $p_1$ and $p_2$, both on side of the cylinder. If the geodesic crosses the boundary of a cap at $x_1$ and $x_2$, then these points need to be chosen such that the curve going through the points $p_1, x_1, x_2$, and $p_2$ is straight in the unfolded picture.
\end{enumerate}
Both of these problems are solved from the optimisation problem
\[
    d(p_1,p_2) = \min_{(x_1,x_2)} \big \{ d(p_1,x_1) + d(x_1,x_2) + d(x_2,p_2) \big \}.
\]
With a bit of trigonometry\footnote{A detailed derivation is included with our implementation \cite{valkonenvonkoch2024nonriemannian-codes}}, this can be solved with several applications of the law of cosines and minimising with respect to the angles where the geodesic crosses the edges.
The resulting solutions are used to form the exponential and logarithm maps on the closed cylinder.
However, these maps can only be solved numerically, as the solutions to the geodesic optimisation problems form transcendental equations.
We illustrate the numerical behaviour of our forward-backward algorithm on the cylinder in \cref{fig:numerical:cylinder}.
These simple numerical experiments validate our theory.

\begin{figure}[t]%
    \input{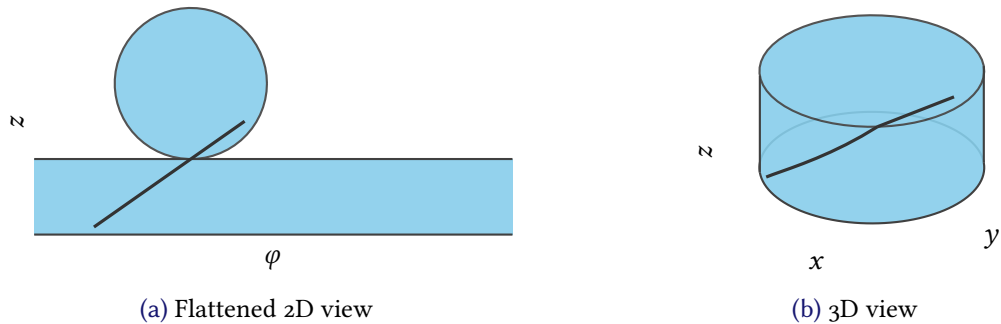}%
    \centering%
    \begin{subfigure}{0.5\linewidth}%
        \centering%
        \tikzsetnextfilename{cyltangent2d}%
        \begin{tikzpicture}
    \begin{axis}[
        set layers,
        axis equal,
        width=\linewidth,
        height=0.6\linewidth,
        enlargelimits=false,
        xlabel = {$\phi$},
        ylabel = {$z$},
        ticks = none,
        axis line style = {draw=none},
        geodesic/.style = {very thick, gray!40!black,forget plot},
    ]

        \addplot[edge, name path=A] coordinates { ({-pi}, {-0.5}) ({pi}, {-0.5}) };
        \addplot[edge, name path=B] coordinates { ({-pi}, { 0.5}) ({pi}, { 0.5}) };
        \addplot3[fill=SkyBlue, very nearly opaque] fill between [of=A and B];

        \addplot[edge, fill=SkyBlue, very nearly opaque, domain={0:2*pi}, samples=60]
            ({\edgephi+cos(deg(x))}, {1.5+sin(deg(x))});

        \addplot[geodesic, domain={0:\edgehit}, samples=2]
            ({\pphi + x*\tphi},
             {\pz + x*\tz});

        \addplot[geodesic, domain={0:\desthit}, samples=2]
            ({\edgephi+x*\tphi},
             {\edgez+x*\tz});

    \end{axis}
\end{tikzpicture}%
        \caption{Flattened 2D view}%
        \label{fig:cyltangent:2d}
    \end{subfigure}%
    \begin{subfigure}{0.5\linewidth}%
        \centering%
        \tikzsetnextfilename{cyltangent3d}%
        \begin{tikzpicture}
    \pgfmathsetmacro{\a}{-pi*5/6}
    \pgfmathsetmacro{\b}{pi/6}

    \begin{axis}[
        set layers,
        axis equal,
        width=0.7\linewidth,
        height=0.7\linewidth,
        enlargelimits=false,
        xlabel = {$x$},
        ylabel = {$y$},
        zlabel = {$z$},
        ticks = none,
        axis line style = {draw=none},
        geodesic/.style = {very thick, gray!40!black,forget plot},
    ]
        \addplot3[edge, domain=\b:(2*pi+\a), samples = 90, samples y = 0, on layer={pre main}] ({cos(deg(x))}, {sin(deg(x))}, {-0.5});

        \addplot3[edge, fill=SkyBlue, very nearly opaque, domain={0:2*pi}, samples = 60, samples y = 0] ({cos(deg(x))}, {sin(deg(x))}, {0.5});
        \addplot3[name path=B,edge, domain={\a:\b}, samples = 60, samples y = 0] ({cos(deg(x))}, {sin(deg(x))}, {-0.5});
        \addplot3[draw=none,name path=A,edge, domain={\a:\b}, samples = 60, samples y = 0] ({cos(deg(x))}, {sin(deg(x))}, {0.5});
        \addplot3[fill=SkyBlue, very nearly opaque] fill between [of=A and B];
        \addplot3[edge] coordinates { (cos(deg(\a)), sin(deg(\a)), -0.5) (cos(deg(\a)), sin(deg(\a)), 0.5) };
        \addplot3[edge] coordinates { (cos(deg(\b)), sin(deg(\b)), -0.5) (cos(deg(\b)), sin(deg(\b)), 0.5) };

        \addplot3[geodesic, domain={0:\edgehit}, samples= 10, samples y = 0]
            ({cos(deg(\pphi + x*\tphi))},
             {sin(deg(\pphi + x*\tphi))},
             {\pz + x*\tz});

        \addplot3[geodesic, domain={0:\desthit}, samples= 10, samples y = 0]
            ({\edgex+x*\trx},
             {\edgey+x*\try},
             {0.5});

    \end{axis}
\end{tikzpicture}%
        \caption{3D view}%
        \label{fig:cyltangent:3d}
    \end{subfigure}%
    \caption{%
        Cylinder tangents and geodesics.
        In the flattened 2D view of (\subref{fig:cyltangent:2d}), geodesics are straight lines; for any two points on the geodesic, the tangent vector given by the logarithmic map is parallel to this line.
        In the 3D view of (\subref{fig:cyltangent:3d}), the same geodesic is a helical path on the side, and again a straight line on the top.
    }
    \label{fig:cyltangent}
\end{figure}

\begin{figure}[t]%
    \centering%
    \tikzexternalenable%
    \begin{subfigure}{0.5\linewidth}%
        \centering%
        \tikzsetnextfilename{cube}%
        \def\datapath{res/cube}%
        \begin{tikzpicture}
    \pgfplotsset{
        onlyfront/.code = \pgfplotsset{x filter/.expression={%
            or(\thisrow{face} == 2, or(\thisrow{face} == 4, \thisrow{face} == 6)) ? x : nan%
        },},
        onlyback/.code = \pgfplotsset{x filter/.expression={%
            or(\thisrow{face} == 1, or(\thisrow{face} == 3, \thisrow{face} == 5)) ? x : nan%
        },},
    }
    
    \begin{axis}[illustr3d]
        \addplot3[edge] coordinates {(0, 0, 0) (0, 1, 0) (0, 1, 1) };
        \addplot3[edge] coordinates {(0, 1, 0) (1, 1, 0) };

        \addplot3[backdata] table[x=x,y=y,z=z] {\datapath/data.csv};

        \addplot3[backiter1] table[x=x,y=y,z=z] {\datapath/x1_log.csv};
        \addplot3[backiter2] table[x=x,y=y,z=z] {\datapath/x2_log.csv};
        \addplot3[backiter3] table[x=x,y=y,z=z] {\datapath/x3_log.csv};

        \addplot3[
            surf,
            mesh/ordering=x varies, 
            mesh/cols=32, 
            surfstyle,
        ] table [
            x = u,
            z = v,
            y expr = 0.0,
            point meta = \thisrow{value},
         ] {\datapath/F4.csv};
        \addplot3[
            surf,
            mesh/ordering=x varies, 
            mesh/cols=32, 
            surfstyle,
        ] table [
            x expr = 1.0,
            y = u,
            z = v,
            point meta = \thisrow{value},
         ] {\datapath/F2.csv};
        \addplot3[
            surf,
            mesh/ordering=x varies, 
            mesh/cols=32, 
            surfstyle,
        ] table [
            x = u,
            y = v,
            z expr = 1.0,
            point meta = \thisrow{value},
        ] {\datapath/F6.csv};
    
        \addplot3[edge] coordinates {(0, 0, 0) (1, 0, 0) (1, 0, 1) (0, 0, 1) (0, 0, 0)};
        \addplot3[edge] coordinates {(0, 0, 1) (0, 1, 1) (1, 1, 1) (1, 0, 1)};
        \addplot3[edge] coordinates {(1, 0, 0) (1, 1, 0) (1, 1, 1)};

        \addplot3[data,onlyfront] table[x=x,y=y,z=z] {\datapath/data.csv};
        \addlegendentry{Data}

        \addplot3[iter1,onlyfront] table[x=x,y=y,z=z] {\datapath/x1_log.csv};
        \addlegendentry{Iterates 1}
        \addplot3[iter2,onlyfront] table[x=x,y=y,z=z] {\datapath/x2_log.csv};
        \addlegendentry{Iterates 2}
        \addplot3[iter3,onlyfront] table[x=x,y=y,z=z] {\datapath/x3_log.csv};
        \addlegendentry{Iterates 3}

        \addplot3[origin,onlyfront] table[x=x,y=y,z=z] {\datapath/origin.csv};
        \addlegendentry{Origin}
    \end{axis}
\end{tikzpicture}%
        \caption{Cube}%
        \label{fig:numerical:cube}
    \end{subfigure}%
    \begin{subfigure}{0.5\linewidth}%
        \centering%
        \tikzsetnextfilename{cylinder}%
        \def\datapath{res/cylinder}%
        \begin{tikzpicture}
    \pgfplotsset{
        cyl data/.code = \pgfplotsset{%
            table/x expr = \thisrow{r} * cos(deg(\thisrow{angle})),
            table/y expr = \thisrow{r} * sin(deg(\thisrow{angle})),
            table/z = z,
        },
        onlyfront/.code = \pgfplotsset{x filter/.expression={%
            \thisrow{face} == 0 || (%
                and(\thisrow{face}==2, or(2*pi+\a <= \thisrow{angle}, \thisrow{angle} <= \b))
            ) ? x : nan%
        },},
        onlyback/.code = \pgfplotsset{x filter/.expression={%
            \thisrow{face} == 1 || (%
                and(\thisrow{face}==2, or(2*pi+\a > \thisrow{angle}, \thisrow{angle} > \b))
            ) ? x : nan%
        },},
    }

    \pgfmathsetmacro{\a}{-pi*5/6}
    \pgfmathsetmacro{\b}{pi/6}

    \begin{axis}[illustr3d]
        \addplot3[edge, domain=\b:(2*pi+\a), samples = 90, samples y = 0] ({cos(deg(x))}, {sin(deg(x))}, {-0.5});

        \addplot3[backdata] table [cyl data] {\datapath/data.csv};

        \addplot3[backiter1,onlyback] table [cyl data] {\datapath/x1_log.csv};
        \addplot3[backiter2,onlyback] table [cyl data] {\datapath/x2_log.csv};
        \addplot3[backiter3,onlyback] table [cyl data] {\datapath/x3_log.csv};

        \addplot3[
            surf,
            mesh/ordering=x varies, 
            mesh/cols=16, 
            surfstyle,
        ] table [
            x expr = cos(deg(\thisrow{angle})),
            y expr = sin(deg(\thisrow{angle})),
            z = v,
            point meta = \thisrow{value},
         ] {\datapath/side_front.csv};
        \addplot3[
            surf,
            mesh/ordering=x varies, 
            mesh/cols=32, 
            surfstyle,
        ] table [
            x expr = \thisrow{v} * cos(deg(\thisrow{angle})),
            y expr = \thisrow{v} * sin(deg(\thisrow{angle})),
            z expr = 0.5,
            point meta = \thisrow{value},
         ] {\datapath/top.csv};
     
        \addplot3[edge, domain={0:2*pi}, samples = 60, samples y = 0] ({cos(deg(x))}, {sin(deg(x))}, {0.5});
        \addplot3[edge, domain={\a:\b}, samples = 60, samples y = 0] ({cos(deg(x))}, {sin(deg(x))}, {-0.5});
        \addplot3[edge] coordinates { (cos(deg(\a)), sin(deg(\a)), -0.5) (cos(deg(\a)), sin(deg(\a)), 0.5) };
        \addplot3[edge] coordinates { (cos(deg(\b)), sin(deg(\b)), -0.5) (cos(deg(\b)), sin(deg(\b)), 0.5) };

        \addplot3[data, onlyfront] table [cyl data] {\datapath/data.csv};
        \addlegendentry{Data}

        \addplot3[iter1,onlyfront] table [cyl data] {\datapath/x1_log.csv};
        \addlegendentry{Iterates 1}
        \addplot3[iter2,onlyfront] table [cyl data] {\datapath/x2_log.csv};
        \addlegendentry{Iterates 2}
        \addplot3[iter3,onlyfront] table [cyl data] {\datapath/x3_log.csv};
        \addlegendentry{Iterates 3}

        \addplot3[origin, onlyfront] table [cyl data] {\datapath/origin.csv};
        \addlegendentry{Origin}
    \end{axis}
\end{tikzpicture}%
        \caption{Cylinder}%
        \label{fig:numerical:cylinder}
    \end{subfigure}
    \caption{%
        Numerical illustrations.
        The surface shading indicates function value.
        Blue dots indicate data points for squared distance terms on exposed faces, while dimmer dots indicate data points on hidden faces. The green circle indicates the
        “origin” data point for the unsquared distance term.
        Algorithm iterates and connecting lines are shown for three different starting points.
    }
    \label{fig:numerical}
    \tikzexternaldisable
\end{figure}
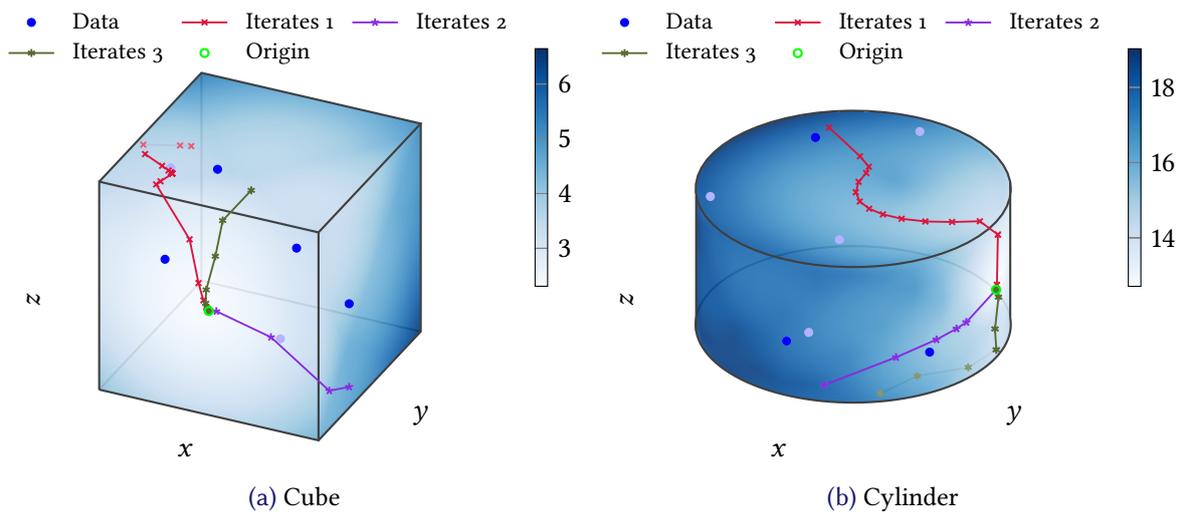

\bibliographystyle{jnsao}

\begin{thebibliography}{10}

\bibitem{petruninfoundations}
S{.\nobreak\kern 0.33333em}Alexander, V{.\nobreak\kern 0.33333em}Kapovitch, and
  A{.\nobreak\kern 0.33333em}Petrunin, Alexandrov geometry: foundations, 2023.

\bibitem{ambrosio2006gradient}
L{.\nobreak\kern 0.33333em}Ambrosio, N{.\nobreak\kern 0.33333em}Gigli, and
  G{.\nobreak\kern 0.33333em}Savare, \emph{Gradient Flows: In Metric Spaces and
  in the Space of Probability Measures}, Lectures in Mathematics. ETH
  Z{\"u}rich, Birkh{\"a}user Basel, 2006.

\bibitem{bacak2014convex}
M{.\nobreak\kern 0.33333em}Ba{\v c}{\'a}k, \emph{Convex Analysis and
  Optimization in Hadamard Spaces}, Nonlinear Analysis and Applications, De
  Gruyter, 2014.

\bibitem{banertbackwardbackward}
S{.\nobreak\kern 0.33333em}Banert, Backward–backward splitting in Hadamard
  spaces, \emph{Journal of Mathematical Analysis and Applications} 414 (2014),
  656--665,
  \href{https://dx.doi.org/https://doi.org/10.1016/j.jmaa.2014.01.054}{\nolinkurl{doi:https://doi.org/10.1016/j.jmaa.2014.01.054}},
  \url{https://www.sciencedirect.com/science/article/pii/S0022247X14000626}.

\bibitem{beier2024tangentialfixpointiterationsgromovwasserstein}
F{.\nobreak\kern 0.33333em}Beier and R{.\nobreak\kern 0.33333em}Beinert,
  Tangential Fixpoint Iterations for Gromov-Wasserstein Barycenters, 2024,
  \href{https://arxiv.org/abs/2403.08612}{\nolinkurl{arXiv:2403.08612}}.

\bibitem{bento2017iteration}
G.\,C{.\nobreak\kern 0.33333em}Bento, O.\,P{.\nobreak\kern 0.33333em}Ferreira,
  and J.\,G{.\nobreak\kern 0.33333em}Melo, Iteration-Complexity of Gradient,
  Subgradient and Proximal Point Methods on Riemannian Manifolds, \emph{Journal
  of Optimization Theory and Applications} 173 (2017),  548--–562,
  \href{https://dx.doi.org/10.1007/s10957-017-1093-4}{\nolinkurl{doi:10.1007/s10957-017-1093-4}}.

\bibitem{berestovskii}
V.\,N{.\nobreak\kern 0.33333em}Berestovskii, Introduction of a Riemann
  structure into certain metric spaces, \emph{Siberian Mathematical Journal} 16
  (1975),  499--507,
  \href{https://dx.doi.org/10.1007/BF00967122}{\nolinkurl{doi:10.1007/bf00967122}}.

\bibitem{Berestovskij1993multidimensional}
V.\,N{.\nobreak\kern 0.33333em}Berestovskij and I.\,G{.\nobreak\kern
  0.33333em}Nikolaev, Multidimensional Generalized Riemannian Spaces, in
  \emph{Geometry IV: Non-regular Riemannian Geometry}, Springer Berlin
  Heidelberg, Berlin, Heidelberg, 1993,  165--243,
  \href{https://dx.doi.org/10.1007/978-3-662-02897-1_2}{\nolinkurl{doi:10.1007/978-3-662-02897-1_2}}.

\bibitem{bergmann2021fenchel}
R{.\nobreak\kern 0.33333em}Bergmann, R{.\nobreak\kern 0.33333em}Herzog,
  M{.\nobreak\kern 0.33333em}Silva~Louzeiro, D{.\nobreak\kern
  0.33333em}Tenbrinck, and J{.\nobreak\kern 0.33333em}Vidal-Núñez, Fenchel
  Duality Theory and a Primal-Dual Algorithm on Riemannian Manifolds,
  \emph{Foundations of Computational Mathematics} 21 (2021),  1465--1504,
  \href{https://dx.doi.org/10.1007/s10208-020-09486-5}{\nolinkurl{doi:10.1007/s10208-020-09486-5}}.

\bibitem{bergmann2025intrinsicriemannianproximalgradient}
R{.\nobreak\kern 0.33333em}Bergmann, H{.\nobreak\kern 0.33333em}Jasa,
  P{.\nobreak\kern 0.33333em}John, and M{.\nobreak\kern 0.33333em}Pfeffer, The
  Intrinsic Riemannian Proximal Gradient Method for Convex Optimization, 2025,
  \url{https://arxiv.org/abs/2507.16055}.

\bibitem{bergmann2016parallel}
R{.\nobreak\kern 0.33333em}Bergmann, J{.\nobreak\kern 0.33333em}Persch, and
  G{.\nobreak\kern 0.33333em}Steidl, A Parallel Douglas–Rachford Algorithm
  for Minimizing {ROF}-like Functionals on Images with Values in Symmetric
  Hadamard Manifolds, \emph{SIAM Journal on Imaging Sciences} 9 (2016),
  901--937,
  \href{https://dx.doi.org/10.1137/15m1052858}{\nolinkurl{doi:10.1137/15m1052858}}.

\bibitem{bridson2011metric}
M{.\nobreak\kern 0.33333em}Bridson and A{.\nobreak\kern 0.33333em}H{\"a}fliger,
  \emph{Metric Spaces of Non-Positive Curvature}, Grundlehren der
  mathematischen Wissenschaften, Springer Berlin Heidelberg, 2011.

\bibitem{burago2001course}
D{.\nobreak\kern 0.33333em}Burago, I{.\nobreak\kern 0.33333em}Burago, and
  S{.\nobreak\kern 0.33333em}Ivanov, \emph{A Course in Metric Geometry}, Crm
  Proceedings \& Lecture Notes, American Mathematical Society, 2001.

\bibitem{yuburagobelow}
Y{.\nobreak\kern 0.33333em}Burago, M{.\nobreak\kern 0.33333em}Gromov, and
  G{.\nobreak\kern 0.33333em}Perel'man, A.{D}. Alexandrov spaces with curvature
  bounded below, \emph{Russian Mathematical Surveys} 47 (1992), ~1,
  \href{https://dx.doi.org/10.1070/RM1992v047n02ABEH000877}{\nolinkurl{doi:10.1070/rm1992v047n02abeh000877}}.

\bibitem{stiefelproxgrad}
S{.\nobreak\kern 0.33333em}Chen, S{.\nobreak\kern 0.33333em}Ma,
  A{.\nobreak\kern 0.33333em}Man-Cho~So, and T{.\nobreak\kern 0.33333em}Zhang,
  Proximal Gradient Method for Nonsmooth Optimization over the Stiefel
  Manifold, \emph{SIAM Journal on Optimization} 30 (2020),  210--239,
  \href{https://dx.doi.org/10.1137/18M122457X}{\nolinkurl{doi:10.1137/18m122457x}},
  \url{https://doi.org/10.1137/18M122457X}.

\bibitem{riemproxgradlinear}
W{.\nobreak\kern 0.33333em}Choi, C{.\nobreak\kern 0.33333em}Chun,
  Y.\,M{.\nobreak\kern 0.33333em}Jung, and S{.\nobreak\kern 0.33333em}Yun, On
  the linear convergence rate of Riemannian proximal gradient method,
  \emph{Optimization Letters} 19 (2025),  667--687,
  \href{https://dx.doi.org/10.1007/s11590-024-02129-6}{\nolinkurl{doi:10.1007/s11590-024-02129-6}},
  \url{https://doi.org/10.1007/s11590-024-02129-6}.

\bibitem{tuomov-nlpdhgm-redo}
C{.\nobreak\kern 0.33333em}Clason, S{.\nobreak\kern 0.33333em}Mazurenko, and
  T{.\nobreak\kern 0.33333em}Valkonen, Acceleration and global convergence of a
  first-order primal-dual method for nonconvex problems, \emph{SIAM Journal on
  Optimization} 29 (2019),  933--963,
  \href{https://dx.doi.org/10.1137/18M1170194}{\nolinkurl{doi:10.1137/18m1170194}},
  \href{https://arxiv.org/abs/1802.03347}{\nolinkurl{arXiv:1802.03347}}.

\bibitem{tuomov2024online-eit}
N{.\nobreak\kern 0.33333em}Dizon, J{.\nobreak\kern 0.33333em}Jauhiainen, and
  T{.\nobreak\kern 0.33333em}Valkonen, Online optimisation for dynamic
  electrical impedance tomography, 2024,
  \href{https://arxiv.org/abs/2412.12944}{\nolinkurl{arXiv:2412.12944}}.
\newblock submitted.

\bibitem{tuomov2024tracking}
N{.\nobreak\kern 0.33333em}Dizon and T{.\nobreak\kern 0.33333em}Valkonen,
  Differential estimates for fast first-order multilevel nonconvex
  optimisation, 2024,
  \href{https://arxiv.org/abs/2412.01481}{\nolinkurl{arXiv:2412.01481}}.
\newblock submitted.

\bibitem{nonconvexriemproxgrad}
S{.\nobreak\kern 0.33333em}Feng, W{.\nobreak\kern 0.33333em}Huang,
  L{.\nobreak\kern 0.33333em}Song, S{.\nobreak\kern 0.33333em}Ying, and
  T{.\nobreak\kern 0.33333em}Zeng, Proximal gradient method for nonconvex and
  nonsmooth optimization on Hadamard manifolds, \emph{Optimization Letters} 16
  (2022),  2277--2297,
  \href{https://dx.doi.org/10.1007/s11590-021-01822-0}{\nolinkurl{doi:10.1007/s11590-021-01822-0}},
  \url{https://doi.org/10.1007/s11590-021-01822-0}.

\bibitem{ferreira2002proximal}
O.\,P{.\nobreak\kern 0.33333em}Ferreira and P.\,R{.\nobreak\kern
  0.33333em}Oliveira, Proximal Point Algorithm On Riemannian Manifolds,
  \emph{Optimization} 51 (2002),  257--270,
  \href{https://dx.doi.org/10.1080/02331930290019413}{\nolinkurl{doi:10.1080/02331930290019413}}.

\bibitem{riemproxgradhuang}
W{.\nobreak\kern 0.33333em}Huang and K{.\nobreak\kern 0.33333em}Wei, Riemannian
  proximal gradient methods, \emph{Mathematical Programming} 194 (2022),
  371--413,
  \href{https://dx.doi.org/10.1007/s10107-021-01632-3}{\nolinkurl{doi:10.1007/s10107-021-01632-3}},
  \url{https://doi.org/10.1007/s10107-021-01632-3}.

\bibitem{Jost1995nonpositivecurvature}
J{.\nobreak\kern 0.33333em}Jost, Convex functionals and generalized harmonic
  maps into spaces of non positive curvature, \emph{Commentarii Mathematici
  Helvetici} 70 (1995),  659--673,
  \href{https://dx.doi.org/10.1007/BF02566027}{\nolinkurl{doi:10.1007/bf02566027}}.

\bibitem{jostnonlineardirichlet}
J{.\nobreak\kern 0.33333em}Jost, Nonlinear Dirichlet Forms, in \emph{New
  Directions in Dirichlet Forms}, AMS/IP studies in advanced mathematics,
  American Mathematical Society, 1998.

\bibitem{kutateladze2005d}
S{.\nobreak\kern 0.33333em}Kutateladze, \emph{A.{D}. Alexandrov: Selected Works
  Part {II}: Intrinsic Geometry of Convex Surfaces}, Classics of Soviet
  Mathematics, CRC Press, 2005.

\bibitem{SobolevKuwae2001}
K{.\nobreak\kern 0.33333em}Kuwae, Y{.\nobreak\kern 0.33333em}Machigashira, and
  T{.\nobreak\kern 0.33333em}Shioya, Sobolev spaces, Laplacian, and heat kernel
  on Alexandrov spaces, \emph{Mathematische Zeitschrift} 238 (2001),  269--316,
  \href{https://dx.doi.org/10.1007/s002090100252}{\nolinkurl{doi:10.1007/s002090100252}},
  \url{https://doi.org/10.1007/s002090100252}.

\bibitem{proxsplitcatlauster}
F{.\nobreak\kern 0.33333em}Lauster and D.\,R{.\nobreak\kern 0.33333em}Luke,
  Convergence of proximal splitting algorithms in {CAT}{$(\kappa)$} spaces and
  beyond, \emph{Fixed Point Theory and Algorithms for Sciences and Engineering}
  2021 (2021), ~13,
  \href{https://dx.doi.org/10.1186/s13663-021-00698-0}{\nolinkurl{doi:10.1186/s13663-021-00698-0}},
  \url{https://doi.org/10.1186/s13663-021-00698-0}.

\bibitem{mayergradientflowsnonpositively}
U.\,F{.\nobreak\kern 0.33333em}Mayer, Gradient flows on nonpositively curved
  metric spaces and harmonic maps, \emph{Communications in Analysis and
  Geometry} 6 (1998),  199--253.

\bibitem{ohtaconvexities}
S.\,I{.\nobreak\kern 0.33333em}Ohta, Convexities of metric spaces,
  \emph{Geometriae Dedicata} 125 (2007),  225--250,
  \href{https://dx.doi.org/10.1007/s10711-007-9159-3}{\nolinkurl{doi:10.1007/s10711-007-9159-3}}.

\bibitem{ohtawasserstein}
S.\,I{.\nobreak\kern 0.33333em}Ohta, Gradient Flows on Wasserstein Spaces over
  Compact Alexandrov Spaces, \emph{American Journal of Mathematics} 131 (2009),
   475--516.

\bibitem{ohta2015discretetime}
S.\,I{.\nobreak\kern 0.33333em}Ohta and M{.\nobreak\kern 0.33333em}Pálfia,
  Discrete-time gradient flows and law of large numbers in Alexandrov spaces,
  \emph{Calculus of Variations}  (2014),
  \href{https://dx.doi.org/10.1007/s00526-015-0837-y}{\nolinkurl{doi:10.1007/s00526-015-0837-y}}.

\bibitem{pressley2010elementary}
A{.\nobreak\kern 0.33333em}Pressley, \emph{Elementary Differential Geometry},
  Springer Undergraduate Mathematics Series, Springer London, 2010.

\bibitem{robbins1971convergence}
H{.\nobreak\kern 0.33333em}Robbins and D{.\nobreak\kern 0.33333em}Siegmund, A
  convergence theorem for non negative almost supermartingales and some
  applications, \emph{Optimizing Methods in Statistics}  (1971),  233–257,
  \href{https://dx.doi.org/10.1016/b978-0-12-604550-5.50015-8}{\nolinkurl{doi:10.1016/b978-0-12-604550-5.50015-8}}.

\bibitem{shahironshortest}
M{.\nobreak\kern 0.33333em}Sharir and A{.\nobreak\kern 0.33333em}Schorr, On
  Shortest Paths in Polyhedral Spaces, \emph{SIAM Journal on Computing} 15
  (1986),  193--215,
  \href{https://dx.doi.org/10.1137/0215014}{\nolinkurl{doi:10.1137/0215014}},
  \url{https://doi.org/10.1137/0215014}.

\bibitem{sturm2023space}
K{.\nobreak\kern 0.33333em}Sturm, \emph{The Space of Spaces: Curvature Bounds
  and Gradient Flows on the Space of Metric Measure Spaces}, Memoirs of the
  American Mathematical Society, American Mathematical Society, 2023.

\bibitem{tuomov-firstorder}
T{.\nobreak\kern 0.33333em}Valkonen, First-order primal-dual methods for
  nonsmooth nonconvex optimisation, in \emph{Handbook of Mathematical Models
  and Algorithms in Computer Vision and Imaging}, K{.\nobreak\kern
  0.33333em}Chen, C.\,B{.\nobreak\kern 0.33333em}Schönlieb,
  X.\,C{.\nobreak\kern 0.33333em}Tai, and L{.\nobreak\kern 0.33333em}Younes
  (eds.), Springer, Cham, 2021,
  \href{https://dx.doi.org/10.1007/978-3-030-03009-4_93-1}{\nolinkurl{doi:10.1007/978-3-030-03009-4_93-1}}.

\bibitem{valkonenvonkoch2024nonriemannian-codes}
T{.\nobreak\kern 0.33333em}Valkonen and H{.\nobreak\kern 0.33333em}von Koch,
  Non-Riemannian optimisation, 2025,
  \href{https://dx.doi.org/10.5281/zenodo.15102168}{\nolinkurl{doi:10.5281/zenodo.15102168}}.
\newblock software on Zenodo.

\bibitem{weinmann2014total}
A{.\nobreak\kern 0.33333em}Weinmann, L{.\nobreak\kern 0.33333em}Demaret, and
  M{.\nobreak\kern 0.33333em}Storath, Total Variation Regularization for
  Manifold-Valued Data, \emph{SIAM Journal on Imaging Sciences} 7 (2014),
  222--2257,
  \href{https://dx.doi.org/10.1137/130951075}{\nolinkurl{doi:10.1137/130951075}}.

\bibitem{zhangfirstorder}
H{.\nobreak\kern 0.33333em}Zhang and S{.\nobreak\kern 0.33333em}Sra,
  First-order Methods for Geodesically Convex Optimization, in \emph{29th
  Annual Conference on Learning Theory}, V{.\nobreak\kern 0.33333em}Feldman,
  A{.\nobreak\kern 0.33333em}Rakhlin, and O{.\nobreak\kern 0.33333em}Shamir
  (eds.), volume~49 of Proceedings of Machine Learning Research, PMLR, Columbia
  University, New York, New York, USA, 2016,  1617--1638.

\end{thebibliography}

\end{document}